\documentclass[sn-mathphys,Numbered]{sn-jnl}
\usepackage{graphicx}%
\usepackage{multirow}%
\usepackage{amsmath,amssymb,amsfonts}%
\usepackage{amsthm}%
\usepackage{mathrsfs}%
\usepackage[title]{appendix}%
\usepackage{xcolor}%
\usepackage{textcomp}%
\usepackage{manyfoot}%
\usepackage{ifthen}
\usepackage[boxed]{algorithm2e}
\usepackage{todonotes}
\usepackage{algpseudocode}%
\usepackage{listings}%
\usepackage{array}
\newtheorem{theorem}{Theorem}
\newtheorem{proposition}[theorem]{Proposition}
\newtheorem{lemma}[theorem]{Lemma}

\newtheorem{remark}{Remark}%
\newtheorem{definition}{Definition}%
\SetKwComment{Comment}{/* }{ */}
\DeclareMathOperator{\var}{var}
\DeclareMathOperator{\prox}{prox}
\DeclareMathOperator{\dom}{dom}
\DeclareMathOperator*{\argmin}{arg\,min}

\DeclareMathOperator{\Span}{span}

\newcommand{\BV}{\textrm{BV}}

\newcommand{\R}{\mathbb{R}}
\newcommand{\N}{\mathbb{N}}

\begin{document}

\title{
  Parameter identification in PDEs by the solution
  of monotone inclusion problems
}
\author*[1]{\fnm{Pankaj} \sur{Gautam}}\email{pgautam908@gmail.com}

\author[2]{\fnm{Markus} \sur{Grasmair}}\email{markus.grasmair@ntnu.no}

\affil[1]{\orgdiv{Department of Applied Mathematics and Scientific Computing}, \orgname{Indian Institute of Technology Roorkee}, \orgaddress{\country{India}}}
\affil[2]{\orgdiv{Department of Mathematical Sciences}, \orgname{Norwegian University of Science and Technology}, \orgaddress{\city{Trondheim}, \country{Norway}}}

\abstract{In this paper we consider the solution of
  monotone inverse problems using the particular
  example of a parameter identification
  problem for a semilinear parabolic PDE.
  For the regularized solution of this problem, we introduce
  a total variation based regularization method requiring
  the solution of a monotone inclusion problem.
  We show well-posedness in the sense of inverse problems
  of the resulting regularization scheme.
  In addition, we introduce and analyze a numerical
  algorithm for the solution of this inclusion problem using a nested inertial
  primal dual method.
  We demonstrate by means of numerical examples the
  convergence of both the numerical algorithm and
  the regularization method.}



\maketitle

\section{Introduction}\label{sec1}

Assume that $U$ is a Hilbert space and that $\mathcal{A} \colon U \to U$
is a possibly non-linear mapping. We consider the inverse problem
of solving the equation
\begin{equation}\label{eq:ip}
  \mathcal{A}(u) = y^\delta,
\end{equation}
given noisy data $y^\delta \in U$ satisfying
\[
  y^\delta = y^\dagger + n^\delta
  \qquad\text{ with } \lVert n^\delta\rVert_{U} \le \delta.
\]
Here $y^\dagger = \mathcal{A}(u^\dagger)$ is the noise free
data produced from the true solution $u^\dagger$.
Specifically, we are interested in the setting
of parameter identification problems, where $\mathcal{A}$
is the solution operator of some given PDE, that is, the operator that maps
the parameters $u$ to the solution of the PDE.
In this case, the problem~\eqref{eq:ip} is typically ill-posed
in that its solution, if it exists at all,
does not depend continuously on the data $y^\delta$.
Because of that, it is necessary to apply some regularization
in order to obtain a stable solution.

In the literature, there exist several approaches
to the regularization of this type of problems, amongst others the following:
\begin{itemize}
\item In Tikhonov regularization,
  one computes an approximate solution of~\eqref{eq:ip} by solving
  the minimization problem
  \begin{equation}\label{eq:Tikh}
    \mathcal{T}_\alpha(u) = \frac{1}{2}\lVert \mathcal{A}(u)-y^\delta\rVert_{U}^2 + \alpha\mathcal{R}(u) \to \min.
  \end{equation}
  Here $\mathcal{R} \colon U \to \R\cup\{+\infty\}$ is a regularization
  term that encodes prior information about the true solution $u^\dagger$,
  and the regularization parameter $\alpha > 0$ steers the trade-off
  between regularity of the solution and data fidelity.
  See e.g.~\cite{EngHanNeu96,SchGraGro09} for an overview and
  analysis of this approach.
\item Iterative regularization methods consider the minimization of the
  norm of the residual $\lVert \mathcal{A}(u)-y^\delta \rVert_{L^2}^2$
  or a similar term by means of an iterative method.
  Examples are Landweber iteration (that is, gradient descent)
  or the iteratively regularized Gauss--Newton or Levenberg--Marquardt method.
  Here the regularization is performed by stopping the iteration early,
  well before convergence.
  An overview of iterative methods in a general setting
  can be found in~\cite{KalNeuSch08}.
\item In the specific case of parameter identification problems, there are
  also ``all-at-once'' formulations, which rewrite the problem~\eqref{eq:ip}
  as a system of equations, the first describing the equation, the second
  the data observation, see e.g.~\cite{Kal16,KunSac92}.
\end{itemize}
In the following, we will discuss a different approach that is most closely related
to the Tikhonov approach.
In order to motivate the method, we note that the necessary optimality
condition for a solution of~\eqref{eq:Tikh} reads
\begin{equation}\label{eq:optcond}
  \mathcal{A}'(u)^*\mathcal{A}(u) + \alpha\partial\mathcal{R}(u) \ni \mathcal{A}'(u)^* y^\delta,
\end{equation}
provided that $\mathcal{A}$ is Fr\'echet differentiable and $\mathcal{R}$ is convex
and lower semi-continuous with subdifferential $\partial\mathcal{R}$.
If $\mathcal{A}$ is bounded linear, the necessary optimality condition is also sufficient,
and the minimizer of $\mathcal{T}_\alpha$ is uniquely characterized by~\eqref{eq:optcond}.
In the non-linear case, however, this is in general not the case,
and there may exist non-optimal solutions of~\eqref{eq:optcond}
as well as local minimizers of $\mathcal{T}_\alpha$.
This makes both the theoretical analysis and numerical implementation
of Tikhonov regularization challenging.
Iterative regularization methods based on the minimization of
the residual $\lVert \mathcal{A}(u)-y^\delta\rVert_{L^2}^2$ face the same challenge,
and most convergence and stability results for these methods hold only
for initializations of the iteration sufficiently close to the true solution.
In addition, the operator $\mathcal{A}$ has to satisfy additional
regularity conditions, a typical example being the \emph{tangential cone condition}
introduced in~\cite[Eq.~1.5]{HanNeuSch95}.

An alternative to Tikhonov regularization that is applicable to
cases where $\mathcal{A}$ is a monotone operator
is \emph{Lavrentiev regularization}, see e.g.~\cite{Tau02,AlbRya06,HofKalRes16}.
The classical formulation requires the
solution of the monotone equation $\mathcal{A}(u) + \alpha u = y^\delta$,
which is related to Tikhonov regularization with the regularization
term $\mathcal{R}(u) = \frac{1}{2}\lVert u \rVert^2_U$.
In \cite{GraHil20}, however, a generalization in the form of the
monotone inclusion problem
\begin{equation}\label{eq:moninc}
  \mathcal{A}(u) + \alpha\partial\mathcal{R}(u) \ni y^\delta
\end{equation}
was proposed, which also allows for the inclusion of non-quadratic regularization
terms similar to Tikhonov regularization.
In contrast to~\eqref{eq:optcond}, this inclusion problem has, under certain
coercivity conditions on $\mathcal{A}$ and $\mathcal{R}$, a unique solution,
and it has been shown in~\cite{GraHil20} that this leads to a stable regularization method.
In this paper, we will develop this approach further and show that it
can be applied to the solution of certain parameter identification problems
for monotone parabolic PDEs with a regularization
term that is a combination of a total variation term in time
and a squared Sobolev norm in space.
Moreover, we will discuss a globally convergent solution algorithm
for the numerical solution of inclusion problems of the form~\eqref{eq:moninc}.

\medskip

In the last decades, there has been a growing interest in the study of monotone inclusion problems within the fields of operator theory and computational optimization. This field of study holds significant relevance, offering practical applications in domains such as partial differential equations, and signal and image processing. The pursuit of identifying the roots of the sum of two or more maximally monotone operators within Hilbert spaces remains a dynamically evolving focus of scientific investigation \cite{BauCom11, Pankaj2021}. Notably, among the methods commonly utilized to address these challenges, splitting algorithms (see \cite[Chapter 25]{BauCom11}) have garnered significant attention.    

Driven by diverse application scenarios, the research community has expressed interest in investigating \textit{primal-dual splitting algorithms} to address complex structured monotone inclusion problem that encompass the presence of finitely many operators, including cases where some of these operators are combined with linear continuous operators and parallel-sum type monotone operators, see~\cite{BotCsetnek2015, Vu2013} and the references therein. The distinguishing feature of these algorithms lies in their complete decomposability, wherein each operator is individually assessed within the algorithm, utilizing either forward or backward steps. 

Primal-dual splitting algorithms, incorporating inertial effects have been featured in \cite{LorenzPock2015, BotCsetnek2016, ChambollePock2016}. These algorithms have demonstrated clear advantages over non-inertial versions in practical experiments \cite{LorenzPock2015, BotCsetnek2016}. The inertial terminology can be noticed as discretization of second order differential equations proposed by Polyak \cite{Polyak1964} to minimize a smooth convex function, the so-called heavy ball method. The presence of an inertial term provides the advantage of using the two preceding terms to determine the next iteration in the algorithm, consequently increases the convergence speed of the algorithm. Nesterov \cite{Nesterov1983} modified the heavy ball method to enhance the convergence rate for smooth convex functions by using the inertial point to evaluate the gradient. In \cite{BeckMarc2009}, Beck and Teboulle have proposed a fast iterative shrinkage-thresholding algorithm (FISTA) within the forward-backward splitting framework for the sum of two convex function, one being non-smooth. The FISTA algorithm is versatile and finds application in numerous practical problems, including sparse signal recovery, image processing, and machine learning.
\bigskip

In this paper, we apply non-linear Lavrentiev regularization
by combining a total variation term in time and a squared Sobolev norm in space to solve
a parameter identification problem for a semi-linear PDE.
This yields a completely new, well-posed regularization method that can be extended
to the solution of more general monotone ill-posed problems.
We discuss the properties of the regularizers and the well-posedness
of the regularization method in Section \ref{se:main_result}.
In addition to showing well-posedness, we discuss the numerical solution of the
regularized problem by providing a numerical algorithm using an inertial technique.
There, we follow the ideas of inexact forward-backward splitting to solve the monotone inclusion problems.
Section \ref{se:well_posedness} provides preliminaries on $L^2$-valued functions of bounded variation
and then provided the proof of well-posedness.
We study convergence of the proposed numerical algorithm in Section \ref{se:proof_algone}.
In Section \ref{se:app_to_our_problem}, we discuss how the numerical method
can be applied to the solution of our inclusion problem.
Finally, in Section~\ref{se:experiments} we present some numerical experiments that show
the behavior of our regularization method as well as the solution algorithm.

\section{Main Results}\label{se:main_result}

Denote by $I := [0,1]$ the unit interval, and let $\Omega\subset \R^d$, $d \in \N$, be
a bounded domain with Lipschitz boundary.
Assume moreover that $\varphi \colon \R \to \R$ is a monotonically
increasing, continuous function satisfying $\varphi(0) = 0$
and $\lim_{s \to \pm \infty} \varphi(s) = \pm \infty$.
Denote by $\mathcal{A} \colon L^2(I\times \Omega) \to L^2(I\times \Omega)$
the operator that maps $u$ to the (weak) solution of the PDE
\begin{equation}\label{eq:PDE}
  \begin{aligned}
    y_t + \varphi(y) - \Delta y &= u &&\text{ in } I \times \Omega,\\
    y &= 0 &&\text{ on } I \times \partial\Omega,\\
    y(0,\cdot) &= y_0 &&\text{ in } \Omega.
  \end{aligned}
\end{equation}
Here $y_0 \in L^2(\Omega)$ is some given function.
We consider the inverse problem of solving the equation
\begin{equation*}
  \mathcal{A}(u) = y^\delta,
\end{equation*}
given noisy data $y^\delta \in L^2(I\times\Omega)$ satisfying
\[
  y^\delta = y^\dagger + n^\delta
  \qquad\text{ with } \lVert n^\delta\rVert_{L^2} \le \delta.
\]
Here $y^\dagger = \mathcal{A}(u^\dagger)$ is the noise free
data produced from the true solution $u^\dagger$.
That is, we want to reconstruct the source term $u^\dagger$
in~\eqref{eq:PDE} from noisy measurements of the associated solution.
We stress here that we consider the case where the source $u^\dagger$
is both space- and time-dependent.

In this paper, we make the specific assumption that the
true solution $u^\dagger$ of~\eqref{eq:ip} is a function that is smooth in the
space variable, but piecewise constant in the time variable.
That is, we can write
\begin{equation}\label{eq:udagger_pw}
  u^\dagger(x,t) = u^\dagger_i(x)
  \qquad\text{ if } t \in [t_{i-1},t_i),
\end{equation}
where $0 = t_0 < t_1 < \ldots < t_N = 1$ is some (unknown) discretization
of the unit interval, and $u_i^\dagger \in H^1(\Omega)$
for $i=1,\ldots,N$.
Because the true solution $u^\dagger$ is piecewise constant
in the time variable, it makes sense to apply some form
of total variation regularization, which is known to promote
piecewise constant solutions.
However, in the spatial direction we want the solutions to
be smooth, which calls for regularization with some type of
Sobolev (semi-)norm.
Thus we will define a regularization term that consists
of the total variation only in the time variable,
and an $H^1$-semi-norm only in the spatial variable.

Denote by
\begin{equation}\label{eq:Rdef}
  \mathcal{R}(u) := \var_t(u) :=
  \sup_{\substack{\eta\in C_0^1(I;L^2(\Omega))\\\lVert \eta(t)\rVert_{L^2} \le 1\text{ for all } t \in I}}\int_0^1 \langle\eta(t),u(t,\cdot)\rangle_{L^2(\Omega)}\,dt
\end{equation}
the total variation of $u$ in the time variable, and by
\[
  \mathcal{S}(u) := \frac{1}{2}\int_I \lvert u(t,\cdot)\rvert_{H^1}^2 \,dt
  = \frac{1}{2}\int_I \int_\Omega\lvert \nabla_x u(t,x) \rvert^2 dx\,dt
\]
the spatial $H^1$-semi-norm of $u$, integrated
over the whole time interval $I$.
We consider the solution of~\eqref{eq:ip} by applying
non-linear Lavrentiev regularization, consisting in the solution
of the monotone operator equation
\begin{equation}\label{eq:problem}
  \mathcal{A}(u) + \partial(\lambda\mathcal{R} + \mu\mathcal{S})(u) \ni y^\delta,
\end{equation}
where $\lambda > 0$ and $\mu > 0$ are regularization parameters
that control the temporal and spatial smoothness of the
regularized solutions, respectively.
Moreover, $\partial(\lambda\mathcal{R}+\mu\mathcal{S})$ is the
subdifferential of the convex and lower semi-continuous function
$\lambda\mathcal{R}+\mu\mathcal{S}\colon L^2(I\times\Omega) \to \R\cup\{+\infty\}$.

\subsection{Well-posedness}

Our first main result states that the solution of~\eqref{eq:problem}
is well-posed in the sense of inverse problems.
That is, for all positive regularization parameters
the solution exists, is unique, and depends continuously on the right hand side $y^\delta$.
Moreover, as the noise level decreases to zero, the solution of~\eqref{eq:problem}
converges to the true solution of the noise free problem~\eqref{eq:ip}
provided the regularization parameters are chosen appropriately.
\smallskip

\begin{theorem}\label{th:well_posed}
  Assume that the solution $u^\dagger$ of the noise-free equation $\mathcal{A}(u) = y^\dagger$
  satisfies $\mathcal{R}(u^\dagger) + \mathcal{S}(u^\dagger) < \infty$.
  The solution of~\eqref{eq:problem} defines a well-posed regularization method.
  That is, the following hold:
  \begin{itemize}
  \item The inclusion~\eqref{eq:problem} admits for each
    $\lambda$, $\mu > 0$ and each $y^\delta$ a unique solution.
  \item Assume that $\lambda$, $\mu > 0$ are fixed,
    and assume that $\{y_k\}_{k\in\N} \in L^2(I\times\Omega)$
    converge to some $y \in L^2(I\times\Omega)$.
    Denote moreover by $u_k$ and $u$ the solutions of~\eqref{eq:problem}
    with right hand sides $y_k$ and $y$, respectively.
    If $\lVert y_k-y\rVert_{L^2} \to 0$, then also $\lVert u_k-u\rVert_{L^2} \to 0$.
  \item 
    Assume that $\lambda = \lambda(\delta)$ and $\mu = \mu(\delta)$
    are chosen such that
    \[
      \lambda(\delta),\ \mu(\delta) \to 0,
      \qquad\text{ and }\qquad
      \frac{\delta}{\lambda(\delta)},\ \frac{\delta}{\mu(\delta)}
      \qquad\text{ are bounded as } \delta \to 0.
    \]
    Denote by $u^\delta_{\lambda,\mu} = u^\delta_{\lambda(\delta),\mu(\delta)}$
    the solution of~\eqref{eq:problem} with right hand side $y^\delta$
    satisfying $\lVert y^\delta - y^\dagger\rVert_{L^2} \le \delta$.
    Then $\lVert u^\delta_{\lambda,\mu} - u^\dagger\rVert_{L^2} \to 0$ as $\delta \to 0$.
  \end{itemize}
\end{theorem}

The proof of this result can be found in Section~\ref{se:wellposed} below.
It mainly relies on a recent general result concerning non-linear
Lavrentiev regularization \cite{GraHil20}.

For the proof of Theorem~\ref{th:well_posed} as well
as the construction of a discretisation of this problem,
we need to reformulate the total variation defined in~\eqref{eq:Rdef}
in a pointwise manner.
For that, we note first that we can identify the
space $L^2(I\times\Omega)$ with the Bochner space
$L^2(I;L^2(\Omega))$ of $L^2$-valued functions on the unit interval $I$.
Thus we can interpret $\mathcal{R}(u)$ as the
total variation of $u$ seen as a function in $L^2(I;L^2(\Omega))$.
Define now 
\[
  \BV(I;L^2(\Omega)) := \bigl\{u \in L^2(I;L^2(\Omega)) : \mathcal{R}(u) < \infty\bigr\}.
\]
It has been shown in~\cite{HeiPetRen19} that
the total variation admits a pointwise interpretation in this space.

\smallskip

\begin{theorem}\label{th:pointwise_BV}
  Assume that $u \in \BV(I;L^2(\Omega))$.
  Then there exists a right-continuous representative $\tilde{u}$
  of $u$ in the sense that
  \[
    \lim_{s \to t^+} \lVert \tilde{u}(s,\cdot)-\tilde{u}(t,\cdot)\rVert_{L^2(\Omega)} = 0
  \]
  for every $t \in [0,1)$.
  Moreover, we have that
  \[
    \mathcal{R}(u) =
    \sup_{0 < t_0 < t_1 < \cdots < t_N < 1} \sum_{i=1}^N \lVert \tilde{u}(t_i,\cdot) - \tilde{u}(t_{i-1},\cdot)\rVert_{L^2(\Omega)}.
  \]
  In addition, $\tilde{u} \colon I \to L^2(\Omega)$
  is continuous outside an at most countable subset of $I$.
\end{theorem}

\begin{proof}
  See~\cite[Prop.~2.1, Prop.~2.3, Cor.~2.11]{HeiPetRen19}.
\end{proof}

In the following we will always identify a function
$u \in \BV(I;L^2(\Omega))$ with its right continuous representative
according to Theorem~\ref{th:pointwise_BV}.
\smallskip

We now discuss the numerical solution of~\eqref{eq:problem}.
First, we note that the domains of both $\mathcal{R}$
and $\mathcal{S}$ are dense proper subspaces of $L^2(I\times\Omega)$,
and $\dom(\mathcal{R})+\dom(\mathcal{S}) \neq L^2(I\times\Omega)$.
Thus we cannot apply results from subdifferential calculus
as found e.g.~in \cite[Sec.~16.4]{BauCom11},
and it is not clear whether the equality
$\partial(\lambda\mathcal{R}+\mu\mathcal{S}) = \lambda\partial\mathcal{R} + \mu\partial\mathcal{S}$
holds.
Because of that and in view of our assumption
concerning the structure of the true solution $u^\dagger$
(see~\eqref{eq:udagger_pw}),
we propose a semi-discretization of~\eqref{eq:problem}.
We fix a grid $\Gamma := \{0=t_0,t_1,\ldots,t_N=1\} \subset I$
with $t_{i-1} < t_i$ for $i=1,\ldots,N$
and denote by $L^2_\Gamma(I\times \Omega)$ the set of functions
$u \in L^2(I\times \Omega)$ such that there exists 
$u_i \in L^2(\Omega)$, $i=1,\ldots,N$, with
\[
  u(t,x) = u_i(x) \qquad\text{ if } t \in [t_{i-1},t_i),\qquad i = 1,\ldots,N.
\]
That is, the functions in $L^2_\Gamma(I\times\Omega)$ are piecewise
constant in the time variable with possible jumps at the
grid points $t_i$.

Define now the operator $D_\Gamma \colon L^2(I\times \Omega) \to L^2(\Omega)^{N-1}$,
\[
  (D_\Gamma u)_i(x) := \frac{1}{t_{i+1}-t_{i}}\int_{t_{i}}^{t_{i+1}} u(x,t)\,dt
  - \frac{1}{t_{i}-t_{i-1}}\int_{t_{i-1}}^{t_{i}} u(x,t)\,dt
\]
for $1 \le i \le N-1$,
and let $\mathcal{R}_\Gamma \colon L^2(\Omega)^{N-1} \to \mathbb{R}$,
\[
  \mathcal{R}_\Gamma(w) := \sum_{i=1}^{N-1} \lVert w_i\rVert_{L^2}.
\]
Define moreover $\mathcal{S}_\Gamma \colon L^2(I\times \Omega) \to \mathbb{R}\cup\{+\infty\}$,
\[
  \mathcal{S}_\Gamma(u) := \begin{cases}
    \mathcal{S}(u) & \text{ if } u \in L^2_\Gamma(I\times\Omega) \cap \dom(\mathcal{S}),\\
    +\infty & \text{ else.}
  \end{cases}
\]      
Instead of~\eqref{eq:problem}, we then consider the semi-discretization
\begin{equation}\label{eq:problem_sd}
  \mathcal{A}(u) + \partial(\lambda\mathcal{R}_\Gamma \circ D_\Gamma + \mu \mathcal{S})(u) \ni y^\delta.
\end{equation}
Note here that $\mathcal{R}(u) = \mathcal{R}_\Gamma(D_\Gamma u)$ and
$\mathcal{S}(u) = \mathcal{S}_\Gamma(u)$ for all $u \in L^2_\Gamma(I\times\Omega)$
(see Theorem~\ref{th:pointwise_BV}, which connects the
weak definition of the variation used in~\eqref{eq:Rdef} to
a pointwise definition).
\smallskip

\begin{lemma}
  We have that
  \[
    \partial(\lambda\mathcal{R}_\Gamma \circ D_\Gamma + \mu \mathcal{S})
    = \lambda D_\Gamma^* \circ \partial\mathcal{R}_\Gamma \circ D_\Gamma + \mu \partial\mathcal{S}_\gamma.
  \]
\end{lemma}

\begin{proof}
  The operator $D_\Gamma\colon L^2(I\times \Omega) \to L^2(\Omega)^{N-1}$ is bounded linear,
  and $\dom(\mathcal{R}_\Gamma) = L^2(\Omega)^{N-1}$.
  Thus we can apply~\cite[Thm.~16.47]{BauCom11}, which proves the assertion.
\end{proof}

As a consequence, we can rewrite~\eqref{eq:problem_sd} as
\begin{equation}\label{eq:problem_sd_b}
  \mathcal{A}(u) + \lambda D_\Gamma^* \partial\mathcal{R}_\Gamma(D_\Gamma u) + \mu \partial\mathcal{S}_\Gamma(u) \ni y^\delta.
\end{equation}

In the following, we will discuss a general algorithmic approach
for the solution of monotone inclusions of the form~\eqref{eq:problem_sd_b}.
Later in Section~\ref{se:app_to_our_problem},
we will discuss the concrete application to our case.

\subsection{Numerical Algorithm}

We now discuss a general algorithm for solving inclusions of the form
\begin{align}\label{eq:problemmodified}
  \text{find }\hat{u}\in \mathcal{U} \text{ such that } 0\in \mathcal{T}(\hat{u})+\partial( f\circ L)(\hat{u})+\partial g(\hat{u}), 
\end{align}
where $\mathcal{U}$ and $\mathcal{V}$ are Hilbert spaces,
$\mathcal{T}\colon\mathcal{U}\to \mathcal{U}$ is a $C$-cocoercive operator, $L \colon \mathcal{U} \to \mathcal{V}$ is bounded linear,
$f \colon \mathcal{U} \to \bar{\mathbb{R}} $, and 
$g \colon \mathcal{V} \to \bar{\mathbb{R}} $ are proper, convex and lower semicontinuous function and $\mathcal{U}$ is a real Hilbert space.  
\smallskip

\begin{definition}
  Let $C > 0$.
  An operator $\mathcal{T}\colon\mathcal{U}\to \mathcal{U}$ is said to be $C$-cocoercive if 
  \[
    \langle \mathcal{T}x-\mathcal{T}y, x-y\rangle\ge C\|\mathcal{T}x-\mathcal{T}y\|^2 ~\forall x,y\in \mathcal{U}.
  \]
\end{definition}

A conventional methodology for addressing problem \eqref{eq:problemmodified} involves the utilization of the forward-backward (FB) splitting method \cite{LionsMercier1979, AbbasAttouch2015,BauCom11}. This method entails the amalgamation of a forward operator with the proximal of function $f\circ L+g$. Consequently, this approach engenders an iterative sequence $\{u_n\}$ following a prescribed form:
\begin{align}\label{eq:FB}
  u_{n+1}=\prox_{\alpha(f\circ L+g)}(u_n-\alpha \mathcal{T}u_n),
\end{align}
where $\alpha>0$ is an appropriate parameter and the proximal operator is defined as 
\[\prox_{\alpha(f\circ L+g)}(x)=\argmin_{u\in \mathcal{U}}\left\{\alpha(f\circ L+g)(u) +\frac{1}{2}\|u-x\|^2\right\}.
\]
The attainment of convergence of sequence \eqref{eq:FB} towards a solution of problem \eqref{eq:problemmodified} can be achieved by making the assumption of cocoerciveness for $\mathcal{T}$ and selecting $\alpha\in (0,2C)$, where $C$ is the cocoercive parameter of $\mathcal{T}$.

Nevertheless, in numerous practical scenarios including the solution of~\eqref{eq:problem_sd}, obtaining the proximal operator for $f\circ L+g$ is not straightforward. Furthermore, a function in the form of $f\circ L$ may not possess a readily available closed-form expression for its proximal operator, which is often the situation with various sparsity-inducing priors used in image and signal processing applications.

In scenarios where explicit access to both the operators ``$\prox g$" and the matrix $L$ is available, the proximal operator can be approximated during each outer iteration through the utilization of an inner iterative algorithm applied to the dual problem. The primary challenge is to determine the optimal number of inner iterations, which profoundly impacts the computational efficiency and theoretical convergence of FB algorithms.

In the literature, two principal strategies have been examined to address the inexact computation of the FB iterate. The first strategy involves developing inexact FB algorithm variants to meet predefined or adaptive tolerance levels. However, these stringent tolerances may lead to a substantial increase in inner iterations and computational costs. Alternatively, another approach prescribes a fixed number of inner iterations, albeit sacrificing some control over the proximal evaluation accuracy. This approach, exemplified in \cite{Chen2019,Bonet2023}, employs a nested primal-dual algorithm, ``warm-started" in each inner loop with results from the previous one. This approach effectively demonstrates convergence towards a solution, even with predetermined proximal evaluation accuracy, as shown in \cite[Theorem 3.1]{Chen2019} and \cite[Theorem 2]{Bonet2023}.

In this section, we present an enhanced variant of the nested primal-dual algorithm that incorporates an inertial step, akin to the FISTA and other Nesterov-type forward-backward algorithms \cite{BeckMarc2009, AttouchJuan2016}, in the setting of the monotone inclusion problem \eqref{eq:problemmodified}. This adaptation can be characterized as an inexact inertial forward-backward algorithm, where the backward step is approximated through a predetermined number of primal-dual iterations and a ``warm–start" strategy for the initialization of the inner loop.

For the numerical solution, we first rewrite \eqref{eq:problemmodified}
as the primal-dual inclusion        
\begin{equation}\label{eq:pdinclusion}
  \text{find } \hat{v}\in \mathcal{V} \text{ such that } (\exists \hat{u}\in \mathcal{U})
  \begin{cases}
    -L^*\hat{v}\in \mathcal{T} \hat{u}+\partial g (\hat{u})\\
    \hat{v}\in \partial f \circ L (\hat{u}).
  \end{cases}
\end{equation}
under the assumption that solutions exist.	
As a next step, we reformulate~\eqref{eq:pdinclusion} further in terms
of fixed points of prox-operators.

\RestyleAlgo{ruled}
\begin{algorithm}
  \caption{Nested inertial primal-dual algorithm}\label{alg:one}
  \SetKwInput{Initialize}{Initialization}
  \Initialize{
    Choose $u_0 = u_{-1}, v_{-1}^{k_{\max}}$, $0<\alpha < 2C$, $0<\beta < 1/\|L\|^2$, $k_{\max}\in \mathbb{N}$, $\{\gamma_n\}\subseteq \mathbb{R}_{\ge 0} $\\
  }
  \BlankLine
  \For{$n=0,1,2,\ldots$}{
    $\bar{u}_n=u_n+\gamma_n(u_n-u_{n-1})$\;
    \text{set}: $v_n^0=v_{n-1}^{k_{\max}}$\;
    \For{$k=0,1,\ldots,k_{{\max}}-1$}{
      $u^k_n= \prox_{\alpha g}(\bar{u}_n-\alpha ( \mathcal{T}(\bar{u}_n) - L^* v_n^k))$\;
      $v_n^{k+1} =\prox_{\beta \alpha^{-1}f^*}(v_n^k+\beta \alpha^{-1}L u_n^k)$\;
    }
    $u_n^{k_{\max}}= \prox_{\alpha g}(\bar{u}_n-\alpha( \mathcal{T}(\bar{u}_n)- L^* v_n^{k_{\max}}))$\;
    ${\displaystyle u_{n+1}=\sum_{k=1}^{k_{\max}}\frac{u_n^k}{k_{\max}}}$\;
  }
\end{algorithm} 

We first recall the following classical result
from~\cite{Moreau1965}, which relates the subdifferential
to the prox-operator.
\smallskip

\begin{lemma}\label{moreau65}
  Let $f\colon\mathcal{U}\to \mathbb{R}$ be a proper, convex and lower semicontinuous function. For all $\alpha,\beta>0$, the following are equivalent:
  \begin{itemize}
  \item [(i)] $u=\prox_{\alpha f}(u+\alpha w)$;
  \item [(ii)] $w\in \partial f(u)$;
  \item [(iii)] $ f(u)+f^*(w)=\langle w, u\rangle $;
  \item [(iv)] $u\in \partial f^*(w)$;
  \item [(v)] $w=\prox_{\beta f^*}(\beta u+w)$.
  \end{itemize}
\end{lemma}

\begin{lemma}
  The solutions $\hat{u}$ of problem \eqref{eq:problemmodified} are characterized by the
  equations
  \begin{align}\label{eq:lemma}
    \begin{cases}
      \hat{u}=\prox_{\alpha g}(\hat{u}-\alpha (\mathcal{T}(\hat{u})-L^*\hat{v}))\\
      \hat{v}=\prox_{\beta \alpha^{-1}f^*}(\hat{v}+\beta \alpha^{-1}L\hat{u})
    \end{cases}
  \end{align}
  for any $\alpha, \beta>0$.
\end{lemma}
\begin{proof}
  If $\hat{u}$ is a solution of \eqref{eq:problemmodified}, then
  there exist $\hat{v}\in \partial f (L\hat{u})$ and $\hat{w}\in \partial g (\hat{u})$
  such that $0=\mathcal{T}(\hat{u})+L^*\hat{v}+\hat{w}$,
  which implies that (using Lemma \ref{moreau65})
  \[
    \hat{u}=\prox_{\alpha g}(\hat{u}+\alpha \hat{w})
    \quad\text{ and }\quad
    \hat{v}=\prox_{\beta \alpha^{-1} f^*}(\hat{v}+\beta \alpha^{-1} L\hat{u}),
  \]
  where $\alpha$, $\beta>0$, and $f^*$ is the Fenchel dual of $f$.
\end{proof}
In the subsequent section, we articulate and substantiate the convergence of the primal-dual sequence produced by Algorithm \ref{alg:one} toward a solution of the problem \eqref{eq:problemmodified}, contingent upon the fulfillment of a suitable technical assumption concerning the inertial parameters.
\smallskip

\begin{theorem}\label{th:convergence}
  Let $\mathcal{T}\colon\mathcal{U}\to \mathcal{U}$ be a $C$-cocoercive operator
  and let $L\colon\mathcal{U}\to \mathcal{V}$ be a bounded linear operator.
  Assume that $f\colon\mathcal{U}\to \bar{\mathbb{R}}$ and
  $g\colon\mathcal{V}\to \bar{\mathbb{R}}$ are proper, convex and lower semicontinuous functions and
  that the sequence $\{\gamma_n\}$ is such that
  \begin{align}
    \gamma := \sum_{n=0}^{\infty}\gamma_n\|u_n-u_{{n-1}}\|<\infty.\label{cond:inertial}
  \end{align}
  If problem \eqref{eq:problemmodified} has a solution, then the sequence $\{(u_n, v_n^0)\}$ generated by Algorithm \ref{alg:one} is bounded and converges to the solution of \eqref{eq:problemmodified}
  
\end{theorem}

The proof of this result can be found in Section~\ref{se:proof_algone}.
\smallskip

\begin{remark}
  For the inertial parameter $\gamma_n$ we propose to
  follow the strategy in \cite{Bonet2023} and define $\gamma_n$ as
  \begin{align*}\gamma_n=
    \begin{cases}
      0, &n=0\\
      \min\left\{\gamma_n^{FISTA}, \frac{\sigma \rho_n}{\|u_n-u_{n-1}\|} \right\}, &n=1,2,\dots
    \end{cases}
  \end{align*}
  where $\sigma>0$ is a constant, $\{\rho_n\}$ is a fixed summable sequence,
  and $\gamma_n^{FISTA}$ is computed according to the usual FISTA rule \cite{BeckMarc2009}
  \begin{align*}t_0=1, \quad
    \begin{cases}
      t_{n+1}&=\dfrac{1+\sqrt{1+4t_n^2}}{2}\\
      \gamma_n^{FISTA}&=\dfrac{t_{n}-1}{t_{n+1}}
    \end{cases} n=0,1\dots
  \end{align*}
  This guarantees that the condition~\eqref{cond:inertial}
  required for the convergence of the iteration holds
  (cf.~the discussion in \cite[Remark 3]{Bonet2023}
  and \cite[p.~318]{LorenzPock2015}).
\end{remark}

\section{Proof of well-posedness}\label{se:well_posedness}

\subsection{Proof of Theorem~\ref{th:well_posed}}\label{se:wellposed}

For the proof of Theorem~\ref{th:well_posed}, we make use
of \cite[Thms.~2.3--2.4]{GraHil20}, where it is shown that
the assertions of the theorem hold, provided that the following
assumptions are satisfied:
\begin{enumerate}
\item\label{it:well_posed1} The underlying space is a Hilbert space.\footnote{In~\cite{GraHil20},
    the more general setting of a reflexive Banach space is used.}
\item\label{it:well_posed2} The operator $\mathcal{A}$ is strictly monotone and hemicontinuous.
\item\label{it:well_posed3} The regularizer $\mathcal{R}(u) + \mathcal{S}(u)$ is proper,
  convex, and lower semi-continuous.
\item\label{it:well_posed4} There exists a solution $u^\dagger$ of the noise-free problem
  such that $\mathcal{R}(u^\dagger) + \mathcal{S}(u^\dagger) < \infty$.
\item\label{it:well_posed5} For all $C > 0$, the sublevel set
  $\{u \in L^2(I\times\Omega) : \lVert u \rVert_{L^2} + \mathcal{R}(u) + \mathcal{S}(u) \le C\}$
  is compact.
\item\label{it:well_posed6} For all sufficiently large $K$ we have that
  \[
    \lim_{r \to \infty} \inf_{u \in U_K} \bigl\langle \mathcal{A}(ru+(1-r)u^\dagger),u-u^\dagger\bigr\rangle = +\infty,
  \]
  where
  \begin{equation}\label{eq:UKdef}
    U_K := \bigl\{u \in L^2(I\times\Omega) : \mathcal{R}(u) + \mathcal{S}(u) \le K \text{ and } \lVert u - u^\dagger \rVert_{L^2} = 1\bigr\}.
  \end{equation}
\end{enumerate}

Assumption~\ref{it:well_posed1} is obviously satisfied.
Concerning the properties of $\mathcal{A}$,
we note that the PDE~\eqref{eq:PDE} can
be equivalently written as the gradient flow
\[
  y_t + \partial \mathcal{G}(y) \ni u,
  \qquad y(0,\cdot) = y_0,
\]
where $\mathcal{G} \colon L^2(\Omega) \to \R \cup\{+\infty\}$ is
the convex and lower semi-continuous functional
\[
  \mathcal{G}(y) = \begin{cases}
    {\displaystyle \int_\Omega \Phi(y)\,dx + \frac{1}{2}\lVert \nabla y\rVert_{L^2}^2}
    & \text{ if } y \in H^1_0(\Omega),\\
    + \infty & \text{ else,}
  \end{cases}
\]
with $\Phi \colon \R \to \R_{\ge 0}$ defined as
\[
  \Phi(s) = \int_0^s \varphi(\xi)\,d\xi.
\]
Thus we can apply the standard theory concerning gradient flows
on Hilbert spaces and obtain the strict monotonicity and (hemi-)continuity
of $\mathcal{A}$ (see e.g.~\cite[Thms.~4.2, 4.5, 4.11]{Bar76}),
which shows Assumption~\ref{it:well_posed2}.
Assumption~\ref{it:well_posed3} follows immediately from the definitions
of $\mathcal{R}$ and $\mathcal{S}$;
Assumption~\ref{it:well_posed4} is one of the assumptions of the theorem.
Assumption~\ref{it:well_posed5} is a direct consequence
of the following result from~\cite{HeiPetRen19}.
\smallskip

\begin{theorem}\label{th:compactsublevelset}
  For every $C > 0$, the sub-level set
  \[
    \bigl\{ u \in L^2(I \times \Omega) : \lVert u \rVert_{L^2}^2 + \mathcal{R}(u) + \mathcal{S}(u) \le C\bigr\}
  \]
  is compact in $L^2(I\times\Omega)$.
\end{theorem}
\begin{proof}
  The embedding $L^2(\Omega) \subset H^1(\Omega)$ is compact.
  Thus we can apply \cite[Thm.~3.22]{HeiPetRen19}
  with $Y = H^1(\Omega)$ and $Z = X^* = L^2(\Omega)$,
  which proves the assertion.
\end{proof}
\bigskip

Thus it remains to verify Assumption~\ref{it:well_posed6},
which is done in Proposition~\ref{pr:well_posed_main} below.
As a preparation of this result, we require several estimates.

We require first a technical result on classical functions
of bounded variation from~\cite{GraHil20}, which states that
such functions can be uniformly bounded away from $0$ on a small interval
with the size of the interval and the bound only depending
on the total variation and the $L^1$-norm of the function.
\smallskip

\begin{lemma}\label{le:bv1d}
  Let $K > 0$ and $d > 0$ be fixed.
  There exist $c > 0$ and $\delta > 0$ such that for all
  $h \in \BV(I)$ with $\lvert Dh \rvert \le K$ and $\lVert h \rVert_1 \ge d$
  there exists an interval $J \subset I$ with $\lvert J \rvert \ge \delta$
  and either $h(t) \ge c$ for all $t \in J$ or $h(t) \le -c$ for all $t \in J$.
\end{lemma}

\begin{proof}
  See~\cite[Lemma~7.3]{GraHil20}.
\end{proof}

Next we show a similar result for projections of functions
in $\BV(I;L^2(\Omega))$.
\smallskip

\begin{lemma}\label{le:coercLe1}
  Let $K > 0$ be fixed and denote
  \[
    W := \bigl\{w \in C^2(\Omega) : \lVert w \rVert_{L^\infty} + \lVert \Delta w \rVert_{L^\infty} \le 1\bigr\}.
  \]
  There exist $c > 0$ and $\delta > 0$ such that
  for every $u \in U_K$ there exist $w \in W$ and an interval $J \subset I$ with
  $\lvert J \rvert \ge \delta$ such that
  \[
    \langle u(t,\cdot) - u^\dagger(t,\cdot),w\rangle_{L^2} \ge c
  \]
  for every $t \in J$.
\end{lemma}

\begin{proof}
  Define $F \colon L^2(I \times \Omega) \to \R_{\ge 0}$,
  \[
    F(u) := \sup_{w \in W} \int_I \lvert \langle u(t,\cdot) - u^\dagger(t,\cdot),w\rangle_{L^2}\rvert\,dt.
  \]
  Then $F$ is convex and lower semi-continuous,
  as it is the supremum of convex and lower semi-continuous functions.
  
  Let now $u \in U_K$, and let $0 < t < 1$ be such that
  $u(t,\cdot) \neq u^\dagger(t,\cdot)$. Here we identify $u$ and $u^\dagger$ with their right continuous
  representatives in $\BV(I;L^2(\Omega))$ (see Theorem~\ref{th:pointwise_BV}).
  Since $\lVert u - u^\dagger\rVert_{L^2} = 1$ and thus $u \neq u^\dagger$, it follows that such a $t$ exists.
  Since $\Span W = C^2(\Omega)$ is dense in $L^2(\Omega)$, there exists some $w \in W$
  such that $\lvert\langle u(t,\cdot)-u^\dagger(t,\cdot),w\rangle_{L^2}\rvert > 0$.
  Because of the right continuity of $u$ and $u^\dagger$ w.r.t.~$t$,
  it follows that also $\int_I \lvert \langle u(t,\cdot)-u^\dagger(t,\cdot),w\rangle_{L^2}\rvert\,dt > 0$.
  As a consequence, $F(u) > 0$ for every $u \in U_K$.
  
  Define now
  \[
    \tilde{c} := \inf_{u \in U_K} F(u).
  \]
  The set $U_K$ is compact in $L^2(I\times\Omega)$ (see Theorem~\ref{th:compactsublevelset}),
  which implies that the minimum in the definition of $\tilde{c}$ is actually attained.
  Since $F(u) > 0$ for all $u$, it follows that $\tilde{c} > 0$.
  As a consequence, there exists for every $u \in U_K$ some $w \in W$
  such that
  \begin{equation}\label{eq:coercLe1_a}
    \int_I \lvert \langle u(t,\cdot) - u^\dagger(t,\cdot),w\rangle_{L^2}\rvert\,dt \ge \frac{\tilde{c}}{2} > 0.
  \end{equation}
  
  Now let $u \in U_K$, and let $w \in W$ be such that~\eqref{eq:coercLe1_a} holds.
  Define the mapping $h \in L^2(I)$,
  \[
    h(t) := \langle u(t,\cdot) - u^\dagger(t,\cdot),w\rangle_{L^2}.
  \]
  Then
  \[
    \lvert Dh \rvert_{\BV} \le \mathcal{R}(u - u^\dagger) \le K + \mathcal{R}(u^\dagger)
    \qquad\text{ and }\qquad
    \int_I \lvert h(t)\rvert\,dt \ge \frac{\tilde{c}}{2}.
  \]
  From Lemma~\ref{le:bv1d}
  it follows that there exist $\delta > 0$ and $c > 0$ only depending
  on $K$, $\mathcal{R}(u^\dagger)$, and $\tilde{c}$ (and thus independent of the choice of $u \in U_K$),
  and an interval $J \subset I$ with $\lvert J \rvert \ge \delta$, such that
  either $h(t) \ge c$ for all $t \in J$ or $h(t) \le -c$ for all $t \in J$.
  After replacing $w$ by $-w$ if necessary, we thus arrive at the claim.
\end{proof}

\begin{lemma}\label{le:coercLe2}
  Let $K > 0$ be fixed.
  There exist constants $C_1 > 0$ and $C_2 \in \R$
  such that for every $u \in U_K$ we have
  \[
    \int_{I\times \Omega} \lvert \mathcal{A}(ru+(1-r)u^\dagger)(t,\cdot)\rvert
    + \lvert \varphi(\mathcal{A}(ru+(1-r)u^\dagger)(t,\cdot))\rvert\,dt\,dx \ge C_1 r - C_2.
  \]
\end{lemma}

\begin{proof}
  Let $u \in U_K$.
  Let $c > 0$, $\delta > 0$, $J \subset I$ and $w \in W$ be as in Lemma~\ref{le:coercLe1}.
  
  For $r > 0$ denote $y_r = \mathcal{A}(ru+(1-r)u^\dagger)$.
  Then
  \begin{multline*}
    \langle ru(t,\cdot)+(1-r)u^\dagger(t,\cdot),w\rangle_{L^2}
    = \langle y_{r,t}(t,\cdot),w\rangle_{L^2}
    + \langle \varphi(y_r)(t,\cdot),w\rangle_{L^2}
    + \langle \nabla y_r(t,\cdot),\nabla w\rangle_{L^2}\\
    =\partial_t \langle y_{r}(t,\cdot),w\rangle_{L^2}
    + \langle \varphi(y_r)(t,\cdot),w\rangle_{L^2}
    - \langle y_r(t,\cdot),\Delta w\rangle_{L^2}
  \end{multline*}
  for almost every $t \in I$.
  Since $\lVert w \rVert_\infty \le 1$ and $\lVert \Delta w \rVert_\infty \le 1$,
  we can further estimate
  \[
    \begin{aligned}
      \langle \varphi(y_r)(t,\cdot),w\rangle_{L^2} & \le \int_\Omega \lvert \varphi(y_r)(t,x)\rvert\,dx,\\
      - \langle y_r(t,\cdot),\Delta w\rangle_{L^2} & \le \int_\Omega \lvert y_r(t,x)\rvert\,dx.
    \end{aligned}
  \]  
  Since by assumption
  \begin{multline*}
    \langle ru(t,\cdot)+(1-r)u^\dagger(t,\cdot),w\rangle_{L^2}
    = r \langle u(t,\cdot)-u^\dagger(t,\cdot),w\rangle_{L^2} + \langle u^\dagger(t,\cdot),w\rangle_{L^2}\\
    \ge cr + \langle u^\dagger(t,\cdot),w\rangle_{L^2}
    \ge cr - \lVert u^\dagger(t,\cdot) \rVert_{L^2},
  \end{multline*}
  we have for almost every $t \in J$ that
  \begin{equation}\label{eq:coercLe2_a}
    cr \le \partial_t \langle y_r(t,\cdot),w\rangle_{L^2} + \int_\Omega \lvert \varphi(y_r)(t,x)\rvert + \lvert y_r(t,x)\rvert\,dt\,dx
    + \lVert u^\dagger(t,\cdot) \rVert_{L^2}.
  \end{equation}
  Now let $t_0 \in J$ be such that the interval $(t_0-\delta/2,t_0+\delta/2)$
  is contained in $J$ (choose e.g.~the mid-point of $J$). Then we obtain
  by integrating~\eqref{eq:coercLe2_a} from $t_0 - s$ to $t_0 + s$ that
  \begin{multline}\label{eq:coercLe2_b}
    2rcs \le \langle y_r(t_0+s,\cdot),w\rangle_{L^2} - \langle y_r(t_0-s,\cdot),w\rangle_{L^2}\\
    + \int_{t_0 - s}^{t_0 + s} \int_\Omega \lvert \varphi(y_r)(t,x)\rvert + \lvert y_r(t,x)\rvert\,dx\,dt + \int_{t_0-s}^{t_0 + s}\lVert u^\dagger(t,\cdot) \rVert_{L^2}\,dt.
  \end{multline}
  for all $0 \le s \le \delta/2$.
  Moreover, as $\lVert w \rVert_\infty \le 1$, we can estimate
  \begin{equation}\label{eq:coercLe2_b1}
    \langle y_r(t_0+s,\cdot),w\rangle_{L^2} - \langle y_r(t_0-s,\cdot),w\rangle_{L^2}
    \le \int_\Omega \lvert y_r(t_0+s,x) \rvert + \lvert y_r(t_0-s,x)\rvert\,dx.
  \end{equation}
  
  Now define for $0 \le s \le \delta/2$
  \[
    G_r(s) := \int_{t_0-s}^{t_0+s} \int_\Omega \lvert \varphi(y_r)(t,x)\rvert + \lvert y_r(t,x)\rvert\,dx\,dt
  \]
  Then~\eqref{eq:coercLe2_b} and \eqref{eq:coercLe2_b1} imply that
  \[
    G_r'(s) \ge 2rcs - f(s) - G_r(s),
  \]
  where
  \[
    f(s) = \int_{t_0-s}^{t_0 + s}\lVert u^\dagger(t,\cdot) \rVert_{L^2}\,dt.
  \]
  Note also that $G_r(0) = 0$.
  
  Now let $H_r$ be the solution of the ODE
  \[
    H_r'(s) = 2rcs - f(s) - H_r(s),
    \qquad\qquad H_r(0) = 0.
  \]
  That is,
  \[
    H_r(s) =
    e^{-s} \int_0^s (2rcp-f(p))e^{p}\,dp
    = 2c(e^{-s} + s - 1)r - e^{-s}\int_0^s f(p)e^p \,dp.
  \]
  Then
  \[
    G_r(s) \ge H_r(s)
    \qquad\text{ for all } 0 \le s \le \delta/2.
  \]
  In particular,
  \begin{multline*}
    \int_{I\times \Omega} \lvert \mathcal{A}(ru+(1-r)u^\dagger)(t,\cdot)\rvert
    + \lvert \varphi(\mathcal{A}(ru+(1-r)u^\dagger)(t,\cdot))\rvert\,dt\,dx \\
    \ge G_r(\delta/2) \ge H_r(\delta/2) = 2c \Bigl(e^{-\delta/2} + \frac{\delta}{2} - 1\Bigr) r - e^{-\delta/2}\int_0^{\delta/2} f(p)e^{p}\,dp,
  \end{multline*}
  which proves the assertion, as $c$, $\delta$, and $f$ were independent of $u$.
  Note here that $e^{-s} \ge 1-s$ for all $s$ in view of the strict
  concavity of the function $s \mapsto e^{-s}$,
  which implies that the factor in the linear term is strictly positive.
\end{proof}

\begin{lemma}\label{le:coercLe3}
  Let $K > 0$ be fixed.
  We have that
  \[
    \lim_{r\to\infty} \frac{1}{r}\inf_{u \in U_K} \bigl\langle \mathcal{A}(ru+(1-r)u^\dagger),ru + (1-r)u^\dagger\bigr\rangle = +\infty.
  \]
\end{lemma}

\begin{proof}
  Let $u \in U_K$ be arbitrary and denote $y_r := \mathcal{A}(ru+(1-r)u^\dagger)$.
  Then
  \begin{multline*}
    \bigl\langle \mathcal{A}(ru+(1-r)u^\dagger),ru + (1-r)u^\dagger\bigr\rangle\\
    = \int_\Omega y_r(1,x)^2\,dx - \int_\Omega y_0(x)^2\,dx
    + \int_{I\times\Omega} \varphi(y_r)(t,x)  y_r(t,x)\,dx\,dt + \int_{I\times\Omega} \lvert \nabla y_r(t,x)\rvert^2\,dx\,dt\\
    \ge  \int_{I\times\Omega} \varphi(y_r)(t,x)  y_r(t,x)\,dx\,dt - \int_\Omega y_0(x)^2\,dx.
  \end{multline*}

  Now define the function $\varphi^- \colon \R_{\ge 0} \to \R_{\ge 0}$,
  \[
    \varphi^-(s) := \min\{\varphi(s),-\varphi(-s)\} = \min\{\lvert\varphi(s)\rvert,\lvert \varphi(-s)\rvert\}.
  \]
  Then $\varphi^-$ is an increasing function satisfying
  $\varphi^-(0) = 0$ and $\lim_{s \to \infty} \varphi^-(s) = +\infty$.
  In particular, the function $\Phi^- \colon \R_{\ge 0} \to \R_{\ge 0}$,
  \[
    \Phi^-(s) := \int_0^s \varphi^-(p)\,dp
  \]
  is convex and superlinearly increasing, and satisfies $\Phi^-(0) = 0$.
  Thus we can estimate
  \[
    s\varphi(s) \ge \lvert s \rvert \varphi^-(\lvert s \rvert) \ge \Phi^-(\lvert s \rvert)
  \]
  for all $s \in \R$.
  As a consequence, Jensen's inequality implies that
  \[
    \int_{I\times \Omega} \varphi(y_r)(t,x) y_r(t,x)\,dx\,dt
    \ge \int_{I\times\Omega} \Phi^-(\lvert y_r(t,x)\rvert)\,dx\,dt
    \ge \Phi^-\Bigl( \int_{I\times\Omega} \lvert y_r(t,x)\rvert\,dx\,dt\Bigr),
  \]
  where we have used the assumption that $\lvert I \times \Omega \rvert = \lvert \Omega \rvert = 1$.

  Next define the function $\varphi^+ \colon \R_{\ge 0} \to \R_{\ge 0}$,
  \[
    \varphi^+(s) := \max\{\varphi(s),-\varphi(-s)\} = \max\{\lvert\varphi(s)\rvert,\lvert \varphi(-s)\rvert\}.
  \]
  Denote moreover by $\psi \colon \R_{\ge 0} \to \R_{\ge 0}$ the lower semi-continuous
  right inverse of $\varphi^+$. That is,
  \[
    \varphi^+(\psi(s)) = s
    \qquad\text{ and }\qquad
    \psi(\varphi^+(s)) \le s
  \]
  for all $s \ge 0$.
  Then, again, $\psi$ is an increasing function satisfying $\psi(0) = 0$
  and $\lim_{s\to\infty} \psi(s) = +\infty$, and thus the function
  $\Psi \colon \R_{\ge 0} \to \R_{\ge 0}$,
  \[
    \Psi(s) := \int_0^s \psi(p) \,dp
  \]
  is also convex and superlinearly increasing, and satisfies $\Psi(0) = 0$. Moreover,
  \[
    s\varphi(s) = \lvert s \rvert \lvert \varphi(s) \rvert
    \ge \psi\bigl(\lvert\varphi^+(s)\rvert\bigr) \lvert \varphi(s) \rvert
    \ge \psi\bigl(\lvert\varphi(s)\rvert\bigr) \lvert \varphi(s) \rvert
  \]
  for all $s \in \R$. 
  Thus Jensen's inequality implies that
  \begin{multline*}
    \int_{I\times \Omega} \varphi(y_r)(t,x) y_r(t,x)\,dx\,dt
    \ge \int_{I\times\Omega} \Psi(\lvert \varphi(y_r)(t,x)\rvert)\,dx\,dt\\
    \ge \Psi\Bigl( \int_{I\times\Omega} \lvert \varphi(y_r)(t,x)\rvert\,dx\,dt\Bigr).
  \end{multline*}

  Next we denote by $P \colon \R_{\ge 0} \to \R_{\ge 0}$,
  \[
    P(s) := (\Phi^- \square \Psi)(s) := \inf_{0 \le p \le s} \Phi^-(p) + \Psi(s-p)
  \]
  the inf-convolution of $\Phi^-$ and $\Psi$.
  Then $P$ is a convex and non-negative function with $P(0) = 0$.
  Moreover, since $\Phi^-(0) = 0$ and $\Psi(0) = 0$, it follows that
  $P(s) \le \Psi(s)$ and $P(s) \le \Phi^-(s)$ for all $s \ge 0$.
  Also, the superlinear growth of $\Phi^-$ and $\Psi$ implies
  that $P(s)$ grows superlinearly as $s \to \infty$ as well.

  Thus, using Lemma~\ref{le:coercLe2} we obtain that
  \begin{multline*}
    \int_{I\times\Omega} \varphi(y_r)(t,x) y_r(t,x)\,dx\,dt\\
    \begin{aligned}
      & \ge \frac{1}{2}\Phi^-\Bigl( \int_{I\times\Omega} \lvert y_r(t,x)\rvert\,dx\,dt\Bigr)
      + \frac{1}{2}\Psi\Bigl( \int_{I\times\Omega} \lvert \varphi(y_r)(t,x)\rvert\,dx\,dt\Bigr)\\
      &\ge \frac{1}{2}P\Bigl( \int_{I\times\Omega} \lvert y_r(t,x)\rvert\,dx\,dt\Bigr)
      +  \frac{1}{2}P\Bigl( \int_{I\times\Omega} \lvert \varphi(y_r)(t,x)\rvert\,dx\,dt\Bigr)\\
      &\ge P\Bigl( \frac{1}{2} \int_{I\times\Omega} \lvert y_r(t,x)\rvert + \lvert\varphi(y_r)(t,x)\rvert\,dx\,dt\Bigr)\\
      &\ge P\Bigl(\frac{1}{2}(C_1 r - C_2)\Bigr)
    \end{aligned}
  \end{multline*}
  provided $r$ is sufficiently large so that $C_1 r - C_2 \ge 0$.
  As a consequence,
  \[
    \bigl\langle \mathcal{A}(ru+(1-r)u^\dagger),ru + (1-r)u^\dagger\bigr\rangle
    \ge P\Bigl(\frac{1}{2}(C_1 r - C_2)\Bigr) - \int_\Omega y_0(x)^2\,dx
  \]
  for sufficiently large $r$.  
  Since $C_1$ and $C_2$ only depend on $K$ and $u^\dagger$
  (but not on $u \in U_K$), the assertion follows.
\end{proof}

\begin{lemma}\label{le:coercLe4}
  There exists a constant $C > 0$ such that
  \[
    \lVert \mathcal{A}(u)\rVert_{L^2(I\times\Omega)}^2
    \le C\bigl(\lVert f \rVert_{L^2}^2 + \lVert u\rVert_{L^2(I\times\Omega)}^2\bigr)
  \]
  for all $u \in L^2(I \times \Omega)$.
\end{lemma}

\begin{proof}
  Denote $y := \mathcal{A}(u)$. Then
  \[
    \langle y(t,\cdot),y_t(t,\cdot)\rangle + \langle \varphi(y)(t,\cdot),y(t,\cdot)\rangle
    + \langle \nabla y(t,\cdot), \nabla y(t,\cdot)\rangle = \langle y(t,\cdot),u(t,\cdot) \rangle
  \]
  for all $t \in I$, and thus
  \[
    \frac{d}{dt}\lVert y(t,\cdot)\rVert_{L^2}^2
    \le \frac{1}{2}\lVert y(t,\cdot)\rVert_{L^2}\lVert u(t,\cdot)\rVert_{L^2}
    \le \frac{1}{4}\lVert y(t,\cdot)\rVert_{L^2}^2 + \frac{1}{4}\lVert u(t,\cdot)\rVert_{L^2}^2
  \]
  for all $t \in I$.
  Since $\lVert y(0,\cdot)\rVert_{L^2}^2 = \lVert f \rVert_{L^2}^2$,
  this implies that
  \[
    \lVert y(t,\cdot)\rVert_{L^2}^2 \le H(t),
  \]
  where $H$ is the solution of the ODE
  \[
    H'(t) = \frac{1}{4}H(t) + \frac{1}{4}\lVert u(t,\cdot)\rVert_{L^2}^2,
    \qquad
    H(0) = \lVert f \rVert_{L^2}^2.
  \]
  A brief computation yields that
  \[
    H(t) = \Bigl(\lVert f \rVert_{L^2}^2 + \frac{1}{4}\int_0^t \lVert u(t,\cdot)\rVert_{L^2}^2 e^{-t/4}\,dt\Bigr)e^{t/4} 
    \le \Bigl(\lVert f \rVert_{L^2}^2 + \frac{1}{4} \lVert u\rVert_{L^2(I\times\Omega)}^2\Bigr)e^{t/4} 
  \]
  and thus
  \[
    \lVert y \rVert_{L^2}^2 \le \int_0^1 H(t)\,dt
    \le \Bigl(\lVert f \rVert_{L^2}^2 + \frac{1}{4} \lVert u\rVert_{L^2(I\times\Omega)}^2\Bigr) \int_0^1 e^{t/4}\,dt,
  \]
  which proves the assertion.
\end{proof}

\begin{proposition}\label{pr:well_posed_main}
  Let $K > 0$ and let $U_K$ be as in~\eqref{eq:UKdef}.
  Then
  \[
    \lim_{r \to \infty} \inf_{u \in U_K} \bigl\langle \mathcal{A}(ru+(1-r)u^\dagger),u-u^\dagger\bigr\rangle = +\infty.
  \]
\end{proposition}

\begin{proof}
  We can write
  \begin{multline*}
    \bigl\langle \mathcal{A}(ru+(1-r)u^\dagger),u-u^\dagger\bigr\rangle\\
    = \frac{1}{r} \bigl\langle \mathcal{A}(ru+(1-r)u^\dagger),ru + (1-r)u^\dagger)\bigr\rangle
    - \frac{1}{r} \bigl\langle\mathcal{A}(ru+(1-r)u^\dagger),u^\dagger\bigr\rangle
  \end{multline*}  
  From Lemma~\ref{le:coercLe3} we obtain that
  \begin{equation}\label{eq:well_posed_main:1}
    \lim_{r \to \infty} \frac{1}{r} \inf_{u \in U_K} \bigl\langle \mathcal{A}(ru+(1-r)u^\dagger),ru+(1-r)u^\dagger\bigr\rangle = +\infty.
  \end{equation}
  On the other hand, we obtain from Lemma~\ref{le:coercLe4} that
  \begin{multline}\label{eq:well_posed_main:2}
    \lim_{r \to \infty} \sup_{u\in U_K} \frac{1}{r}\lvert\langle \mathcal{A}(ru+(1-r)u^\dagger),u^\dagger\rangle\rvert\\
    \begin{aligned}
      &\le \lim_{r \to \infty} \sup_{u \in U_K} \frac{1}{r} \Bigl(C\bigl(\lVert f \rVert_{L^2}^2
      + \lVert ru+(1-r)u^\dagger\rVert_{L^2}^2\bigr)\Bigr)^{1/2}\lVert u^\dagger\rVert_{L^2} \\
      &\le \lim_{r \to \infty} \frac{1}{r}\Bigl(C\bigl(\lVert f \rVert_{L^2}^2 + 2r^2 + 2\lVert u^\dagger\rVert_{L^2}^2\bigr)\Bigr)^{1/2}\lVert u^\dagger \rVert_{L^2}\\
      & = \sqrt{2C}\lVert u^\dagger\rVert_{L^2} < \infty.
    \end{aligned}
  \end{multline}
  Here we have used that $\lVert u - u^\dagger \rVert_{L^2} = 1$ for all $u \in U_K$.
  Together, these estimates prove the assertion.
\end{proof}

\section{Convergence of Algorithm~\ref{alg:one}}\label{se:proof_algone}

In this section, we systematically ascertain the convergence properties pertaining to the primal-dual sequences, which are iteratively produced by the execution of Algorithm \ref{alg:one}, with the objective of reaching the optimal solution for inclusion problem \eqref{eq:problemmodified}. This analysis is conducted under the condition that the inertial parameters conform to a set of requisite technical assumptions.
The proof of this result is to a large degree a combination
of the proofs of \cite[Theorem 3.1]{Chen2019} and \cite[Theorem 2]{Bonet2023}.

For the remainder of this section, we denote by $(\hat{u},\hat{v})$ the primal-dual solution of~\eqref{eq:problemmodified}.

\subsection{Convergence: proof of Theorem \ref{th:convergence}}

We start by following the proof of~\cite[Theorem 3.1]{Chen2019},
where the convergence of Algorithm~\ref{alg:one} is shown for
the case where $\gamma_n = 0$ for all $n$ and $\mathcal{T}(u) = \nabla h(u)$
for a convex and differentiable function $h \colon \mathcal{U} \to \R$ with Lipschitz continuous
gradient.

Applying the same algebraic manipulations as in \cite{Chen2019},
we obtain (cf.~(25) in \cite{Chen2019})
\begin{multline*}
  \sum_{k=0}^{k_{\max}-1}\left(\beta \|u_n^{k+1}-\hat{u}\|^2+\alpha^2\|v_n^{k+1}-\hat{v}\|^2\right)\\
  \begin{aligned}    
    &\le \sum_{k=0}^{k_{\max}-1}\Bigr(\beta \|\bar{u}_n-\hat{u}\|^2-\beta\lVert u_n^k - u_n^{k+1}\rVert^2-\beta\lVert u_n^k-\bar{u}_n\rVert^2\\
    &{}\qquad\qquad  -2\alpha\beta\langle u_n^k-\hat{u}, \mathcal{T}(\bar{u}_n)-\mathcal{T}(\hat{u})\rangle + \alpha^2\lVert v_n^k - \hat{v}\rVert^2 - \alpha^2 \lVert v_n^{k+1} - v_n^k\rVert^2\\
    &{}\qquad\qquad +2\alpha\beta\langle u_n^k - u_n^{k+1}, L^*(v_n^{k+1}-v_n^k)\rangle\Bigr).
  \end{aligned}
\end{multline*}
The cocoercivity of $\mathcal{T}$ now lets us estimate
\begin{multline*}
  -\lVert u_n^k-\bar{u}_n\rVert^2-2\alpha\langle u_n^k-\hat{u}, \mathcal{T}(\bar{u}_n)-\mathcal{T}(\hat{u})\rangle\\
  \begin{aligned}
    &=\alpha^2 \lVert \mathcal{T}(\bar{u}_n) - \mathcal{T}(\hat{u})\rVert^2 - 2\alpha \langle \bar{u}_n-\hat{u},\mathcal{T}(\bar{u}_n)-\mathcal{T}(\hat{u})\rangle \\
    &{}\qquad\qquad -\|u_n^k-\bar{u}_n+\alpha (\mathcal{T}(\bar{u}_n)-\mathcal{T}(\hat{u}))\|^2\\
    &\le \alpha(\alpha -2C) \|\mathcal{T}(\bar{u}_n)-\mathcal{T}(\hat{u})\|^2-\|u_n^k-\bar{u}_n+\alpha (\mathcal{T}(\bar{u}_n)-\mathcal{T}(\hat{u}))\|^2.
  \end{aligned}
\end{multline*}
We thus obtain the estimate (cf.~(27) of \cite[Theorem 3.1]{Chen2019})
\begin{align*}
  &\sum_{k=0}^{k_{\max}-1}\bigl(\beta \|u_n^{k+1}-\hat{u}\|^2+\alpha^2\|v_n^{k+1}-\hat{v}\|^2\bigr)\\
  &\quad \le \sum_{k=0}^{k_{\max}-1}\Bigl(\beta \|\bar{u}_n-\hat{u}\|^2+\alpha^2\|v_n^k-\hat{v}\|^2+\alpha\beta(\alpha -2C) \|\mathcal{T}(\bar{u}_n)-\mathcal{T}(\hat{u})\|^2\\
  &\qquad -\beta\|u_n^k-\bar{u}_n+\alpha (\mathcal{T}(\bar{u}_n)-\mathcal{T}(\hat{u}))\|^2-\alpha^2\|v_n^{k+1}-v_n^k\|^2\\
  & \qquad +\alpha^2\beta \|L^*(v_n^{k+1}-v_n^k)\|^2-\beta \|u_n^k-u_{n}^{k+1}-\alpha L^*(v_n^{k+1}-v_n^k)\|^2\Bigr).
\end{align*}
Inserting the inertial term $\bar{u}_n= u_n+\gamma_n(u_n-u_{n-1})$ and using the Cauchy--Schwarz inequality in the above estimate, we deduce
\begin{align*}
  &\sum_{k=0}^{k_{\max}-1}\bigl(\beta \|u_n^{k+1}-\hat{u}\|^2+\alpha^2\|v_n^{k+1}-\hat{v}\|^2\bigr)\\
  &\quad \le \sum_{k=0}^{k_{\max}-1}\Bigl(\beta( \|\bar{u}_n-\hat{u}\|^2+\gamma_n^2\|u_n-u_{n-1}\|^2+2\gamma_n\|u_n-\hat{u}\|\|u_n-u_{n-1}\|)\\
  & \qquad +\alpha^2\|v_n^k-\hat{v}\|^2+\alpha\beta(\alpha -2C) \|\mathcal{T}(\bar{u}_n)-\mathcal{T}(\hat{u})\|^2\\
  &\qquad -\beta\|u_n^k-\bar{u}_n+\alpha (\mathcal{T}(\bar{u}_n)-\mathcal{T}(\hat{u}))\|^2-\alpha^2\|v_n^{k+1}-v_n^k\|^2\\
  & \qquad +\alpha^2\beta \|L^*(v_n^{k+1}-v_n^k)\|^2-\beta \|u_n^k-u_{n}^{k+1}-\alpha L^*(v_n^{k+1}-v_n^k)\|^2\Bigr).
\end{align*}
By the convexity of $\|u-\hat{u}\|^2$ (as a function of $u$) and the last line of Algorithm \ref{alg:one}, we get
\begin{align}
  \beta &k_{max}\|u_{n+1}-\hat{u}\|^2+\alpha^2\|v_{n+1}^0-\hat{v}\|^2\nonumber\\
        &\le \beta k_{\max}( \|u_n-\hat{u}\|^2+\gamma_n^2\|u_n-u_{n-1}\|^2+2\gamma_n\|u_n-\hat{u}\|\|u_n-u_{n-1}\|)\nonumber\\
        & \quad+\alpha^2 \|v_n^0-\hat{v}\|^2 +\alpha \beta (\alpha -2C)\|\mathcal{T}(\bar{u}_n)-\mathcal{T}(\hat{u})\|^2\nonumber\\
        &\quad -\sum_{k=0}^{k_{\max}-1}\Bigl(\beta\bigl\lVert u_n^k-\bar{u}_n+\alpha (\mathcal{T}(\bar{u}_n)-\mathcal{T}(\hat{u}))\bigr\rVert^2+\alpha^2\|v_n^{k+1}-v_n^k\|_L^2\nonumber\\
        & \quad +\beta \|u_n^k-u_{n}^{k+1}-\alpha L^*(v_n^{k+1}-v_n^k)\|^2\Bigr),\label{eq1:thm14}
\end{align}
where $\|\cdot\|^2_L=\|\cdot\|^2-\beta \|L^*(\cdot)\|^2$.
Since the terms inside the summation of above inequality are all positive, we have
\begin{align}
  \beta &k_{max}\|u_{n+1}-\hat{u}\|^2+\alpha^2\|v_{n+1}^0-\hat{v}\|^2\nonumber\\
        &\le \beta k_{\max} \|u_n-\hat{u}\|^2+\alpha^2 \|v_n^0-\hat{v}\|^2\nonumber\\
        &\quad +\beta k_{\max}(\gamma_n^2\|u_n-u_{n-1}\|^2+2\gamma_n\|u_n-\hat{u}\|\|u_n-u_{n-1}\|),\label{eq:bounded}
\end{align}
which is same as (33) of \cite[Theorem 2]{Bonet2023}.
We now follow that proof further and obtain that the sequence
$\{(u_n,v_n^0)\}$ is bounded and that the sequence
$\{\beta k_{\max} \lVert \hat{u}-u_n\rVert^2 + \alpha^2 \lVert\hat{v}-v_n^0\rVert^2\}$
converges. 
By boundedness of $\{(u_n,v_n^0)\}$ and convergence of $\{\beta k_{\max} \lVert \hat{u}-u_n\rVert^2 + \alpha^2 \lVert\hat{v}-v_n^0\rVert^2\}$, there exists a  point $(u^{\dag}, v^{\dag})$ such that $(u_{n_j},v_{n_j}^0)\to (u^{\dag}, v^{\dag})$ as $j\to \infty$.

Now, first summing the relation \eqref{eq1:thm14} from $n=0$ to $n=N$ and then taking the limit $N\to \infty$ and using the condition \eqref{cond:inertial}, we observe the following:
\begin{align*}
  \lim_{n\to \infty}\|\mathcal{T}(\bar{u}_n)-\mathcal{T}(\hat{u})\|^2 &= 0,\\
  \lim_{n\to \infty} \|u_n^k-\bar{u}_n+\alpha (\mathcal{T}(\bar{u}_n)-\mathcal{T}(\hat{u}))\|^2&=0, &\qquad &k:0,1, \dots, k_{\max}-1,\\
  \lim_{n\to \infty} \|v_n^{k+1}-v_n^k\|_{L}^2&= 0, &\qquad& k:0,1,\dots, k_{\max}-1,\\
  \lim_{n\to \infty} \|u_n^k-u_{n}^{k+1}-\alpha L^*(v_n^{k+1}-v_n^k)\|^2 &=0, &\qquad& k:0,1,\dots, k_{\max}-1.
\end{align*}
The above estimates consequently imply that
\begin{align*}
  u_{n_j}^{k+1}\to u^{\dag}, \qquad v_{n_j}^{k+1}\to v^{\dag} \qquad k:0, 1, \dots, k_{\max}-1,~ \text{as}~ j\to \infty.
\end{align*}
Due to the continuous nature of the proximal operation  of the Algorithm \ref{alg:one}, it can be deduced that $(u^{\dag}, v^{\dag})$ adheres to \eqref{eq:lemma}, which defines the solution of the problem \eqref{eq:problemmodified}. Since $(u^{\dag}, v^{\dag})$ is a saddle point and convergence of the sequence $\{\beta k_{\max}\|u_n-u^{\dag}\|^2+\alpha^2\|v_n^0-v^{\dag}\|\}$, it admits a subsequence converging to zero. Hence the sequence $\{(u_n, v_n^0)\}$ converges to $(u^{\dag}, v^{\dag})$.

\section{Application to \eqref{eq:problem}}\label{se:app_to_our_problem}

We now discuss how to apply Algorithm~\ref{alg:one}
to the solution of~\eqref{eq:problem_sd_b}.
That is, we use Algorithm~\ref{alg:one} with
$\mathcal{T}(u) = \mathcal{A}(u) - y^\delta$,
$f = \lambda \mathcal{R}_\Gamma$, $g = \mu\mathcal{S}_\Gamma$,
and $L = D_\Gamma$.

In order to apply the convergence result Theorem~\ref{th:convergence},
we first have to verify that the operator
$\mathcal{A}$ is cocoercive.

\begin{lemma}
  The operator $\mathcal{A}$ is cocoercive.
  That is, there exists $C > 0$ such that
  \[
    \langle \mathcal{A}(u)-\mathcal{A}(v),u-v\rangle_{L^2}
    \ge C \lVert \mathcal{A}(u)-\mathcal{A}(v)\rVert_{L^2}^2
  \]
  for all $u$, $v \in L^2(I \times \Omega)$.
\end{lemma}

\begin{proof}
  Let $u$, $v \in L^2(I\times \Omega)$ and denote
  $y = \mathcal{A}(u)$ and $z = \mathcal{A}(v)$.
  Then
  \begin{multline*}
    \langle \mathcal{A}(u)-\mathcal{A}(v),u-v\rangle_{L^2}
    = \int_{I \times \Omega} (y-z)(u-v)\,d(t\otimes x)\\
    = \int_\Omega \int_I (y-z)(y_t-z_t)\,dt\,dx
    + \int_{I\times \Omega} (\varphi(y)-\varphi(z))(y-z)\,d(t\otimes x)
    + \int_{I \times \Omega} \lVert \nabla(y-z)\rVert^2\,d(t\otimes x)\\
    \ge \frac{1}{2}\int_\Omega (y(1,x)-z(1,x))^2\,dx
    + 0 + \int_{I \times \Omega} \lVert \nabla(y-z)\rVert^2\,d(t\otimes x).
  \end{multline*}
  Moreover, we obtain from the Poincar\'e inequality
  for the set $\Omega$ that
  \[
    \int_\Omega \lVert \nabla y(t,x) - \nabla z(t,x)\rVert^2\,dx
    \ge C \int_\Omega (y(t,x)-z(t,x))^2\,dx
  \]
  for some $C > 0$ and almost every $t \in I$.
  Thus
  \[
    \langle \mathcal{A}(u)-\mathcal{A}(v),u-v\rangle_{L^2}
    \ge C\int_I\int_\Omega (y(t,x)-z(t,x))^2\,dx\,dt
    = C\lVert y-z\rVert_{L^2}^2,
  \]
  where $C$ is the constant from the Poincar\'e inequality
  for the set $\Omega$.
\end{proof}

\begin{remark}
  In the one-dimensional case with $\Omega = [0,1] \subset \mathbb{R}$,
  it is known that the optimal constant in the Poincar\'e
  inequality is $C = \pi^2$.
\end{remark}

Next, we will provide explicit formulas for
the prox-operators that have to be evaluated in each step of the
algorithm.
For that, we will identify a function
$u \in L^2_\Gamma(I\times\Omega)$ with the $N$-tuple
$(u_i)_{i=1,\ldots,N} \in L^2(\Omega)^N$ satisfying
\[
  u(t,x) = u_i(x) \qquad\text{ if } t \in [t_{i-1},t_i),\qquad i = 1,\ldots,N.
\]
The function $\prox_{\alpha\mu\mathcal{S}}(w)$ is defined as
\[
  \prox_{\alpha\mu\mathcal{S}}(w)
  =
  \min_{u \in L^2_\Gamma(I\times\Omega)} \Bigl[\frac{1}{2}\int_I \int_\Omega (u(t,x)-w(t,x))^2 + \alpha\mu \lvert \nabla_x u(t,x)\rvert^2\,dx\,dt\Bigr].
\]
Since we are only taking derivatives in the space variable,
the expression to be minimized includes no
coupling between the different times $t$ apart from the
requirement that $u \in L^2_\Gamma(I\times \Omega)$,
and thus we can minimize it separately on each strip $[t_{i-1},t_i) \times \Omega$.
Identifying the restriction of $u$ to this strip with $u_i$ we therefore obtain the problem
\[
  \min_{u_i \in L^2(\Omega)} \Bigl[\frac{1}{2}\int_{t_{i-1}}^{t_i}\int_\Omega (u_i(x)-w(t,x))^2 + \alpha\mu \lvert \nabla u_i(x)\rvert^2\,dx\,dt\Bigr]
\]
for $i =1,\ldots,N$.
The first order optimality condition (or Euler--Lagrange equation) for this
problem is the condition that $u_i \in H^1(\Omega)$ and
\[
  \int_{t_{i-1}}^{t_i} \int_\Omega (u_i(x)-w(t,x)) v(x) + \alpha\mu \nabla u_i(x)\cdot \nabla v(x)\,dx\,dt = 0
\]
for all $v \in H^1(\Omega)$.
This is the weak form of the equation
\[
  \begin{aligned}
    (t_i-t_{i-1}) \bigl(u_i(x) - \alpha\mu\Delta u_i(x)\bigr) &= \int_{t_{i-1}}^{t_i} w(t,x)\,dt &&\text{ for } x \in \Omega\\
    \partial_\nu u_i(x) &= 0 &&\text{ for } x \in \partial\Omega.
  \end{aligned}
\]
Thus
\[
  (\prox_{\alpha\mu\mathcal{S}}(w))_i = (I-\alpha\mu \Delta)^{-1}\Bigl(\frac{1}{t_i-t_{i-1}}\int_{t_{i-1}}^{t_i} w(t,x)\,dt\Bigr),
\]
where we solve the PDE with homogeneous Neumann boundary conditions on $\partial\Omega$.

Next, we note that $\mathcal{R}_\Gamma(u) = \sum_{i=1}^{N-1} \lVert u_i \rVert_{L^2}$
is positively homogeneous.
Thus we have that
\[
  (\lambda\mathcal{R}_\Gamma)^*(v)
  = \begin{cases}
    + \infty & \text{ if } \lVert v_i\rVert_{L^2} > \lambda \text{ for some } 1\le i \le N-1,\\
    0 & \text{ if } \lVert v_i \rVert_{L^2} \le \lambda \text{ for all } 1 \le i \le N-1.
  \end{cases}
\]
As a consequence, the prox-operator $\prox_{\beta\alpha^{-1}(\lambda\mathcal{R}_\Gamma)^*}$
is a componentwise projection onto the ball of radius $\lambda$ in $L^2(\Omega)$,
that is,
\[
  (\prox_{\beta\alpha^{-1} (\lambda\mathcal{R})^*}(v))_i
  =
  \left\{
    \begin{aligned}
      \frac{\lambda v_i}{\lVert v_i\rVert_{L^2}} &\qquad\text{ if } \lvert v_i \rVert_{L^2} > \lambda\\
      v_i &\qquad\text{ if } \lVert v_i \rVert_{L^2} \le \lambda
    \end{aligned}
  \right\}
  = \frac{\lambda v_i}{\max\{\lambda,\lVert v_i \rVert_{L^2}\}}.
\]

Finally, we see that $D_\Gamma^* \colon L^2(\Omega)^{N-1} \to L^2(I\times\Omega)$
actually maps into the subspace $L^2_\Gamma(\Omega)$, and we have that
\[
  (D_\Gamma^* v)_i =
  \begin{cases}
    \displaystyle{-\frac{1}{t_1-t_0} v_1} & \text{ if } i=1,\\
    \displaystyle{\frac{1}{t_{i}-t_{i-1}} v_i - \frac{1}{t_{i+1}-t_i} v_{i+1}} & \text{ if } 2 \le i \le N,\\
    \displaystyle{\frac{1}{t_N-t_{n-1}}v_N} & \text{ if } i=N,
  \end{cases}
\]
where we have again identified the function
$(D_\Gamma^* v) \in L^2_\Gamma(I \times \Omega)$ with
and $N$-tuple in $L^2(\Omega)^N$.

The resulting method is summarized in Algorithm~\ref{alg:applied}.

\begin{algorithm}
  \caption{Nested inertial primal-dual algorithm for \eqref{eq:problem_sd}}\label{alg:applied}
  \SetKwInput{Initialize}{Initialization}
  \Initialize{
    Choose $u \in L^2(\Omega)^N$, $v \in L^2(\Omega)^{N-1}$,
    $0<\alpha < 2C$, $0<\beta < 1/\lVert D_\Gamma \rVert^2$, $k_{\max}\in \mathbb{N}$, $\{\gamma_n\}\subseteq \mathbb{R}_{\ge 0} $\;
    Set $u^{(\text{old})} = 0$\;
  }
  \BlankLine
  \For{$n=1,2,\ldots$}{
    $\bar{u} \leftarrow u + \gamma_n(u-u^{(\textrm{old})})$\;
    $w \leftarrow \mathcal{A}(\bar{u}) - y^\delta$\;
    \For{$i=1,\ldots,N$}{
      ${\displaystyle w_i \leftarrow \frac{1}{t_i-t_{i-1}} \int_{t_{i-1}}^{t_i} w(t,x) - y^\delta(t,x)\,dt}$\;
    }
    \For{$k=0,1,\ldots,k_{{\max}}-1$}{
      \For{$i=1,\ldots,N$}{
        $u_i^{(k)} \leftarrow (I-\alpha\mu\Delta)^{-1}(\bar{u}_i - \alpha(w_i + (D_\Gamma^* v)_i))$\;
      }
      \For{$i=1,\ldots,N-1$}{
        $v_i \leftarrow \lambda\dfrac{v_i + \beta\alpha^{-1} (D_\Gamma u^{(k)})_i}{\max\{\lambda,\lVert v_i + \beta\alpha^{-1} (D_\Gamma u^{(k)})_i \rVert\}}$\;
      }
    }
    \For{$i=1,\ldots,N$}{
      $u_i^{(k_{\text{max}})} \leftarrow (I-\alpha\mu\Delta)^{-1}(\bar{u}_i - \alpha(w_i + (D_\Gamma^* v)_i))$\;
    }
    $u^{(\text{old})} \leftarrow u$\;
    ${\displaystyle u \leftarrow \frac{1}{k_{\max}} \sum_{k=1}^{k_{\max}} u^{(k)}}$\;
  }
\end{algorithm} 

\section{Numerical Experiments}\label{se:experiments}

We have tested the method in the one-dimensional case $\Omega = [0,1]$.
For the true solution of the inverse problem we have considered two examples,
first the function
\begin{equation}\label{eq:u1}
  u_1^\dagger(x,t) = \begin{cases}
    4 \sin(\pi x) & \text{ if } 0 < t < \frac{1}{4},\\
    2\cos(7\pi x) & \text{ if } \frac{1}{4} < t < \frac{2}{3},\\
    5 x-\cos(\pi x) & \text{ if } \frac{2}{3} < t < \frac{3}{4},\\
    5 - 4\sin(\pi x) & \text{ if } \frac{3}{4} < t < 1,
  \end{cases}
\end{equation}
then the function
\begin{equation}\label{eq:u2}
  u_2^\dagger = \sin(2\pi t)^5 \cos(2\pi x).
\end{equation}
For the initial condition for the PDE we used the function $y_0 = 0$.
For the non-linearity we have used the function $\varphi(y) = y^3$.

All calculations have been performed in python using the NumPy and SciPy libraries
and running on a MacBook Pro equipped with an Apple M1 processor and 16GB of RAM.

The function $u_1^\dagger$ satisfies the assumptions that
our regularization method makes, in that $u_1^\dagger$ is piecewise
constant in the time variable, but not in the space variable,
where we have significant, but smooth, variations.
In contrast, the function $u_2^\dagger$ is smooth both in time and space,
though there are still large regions, where the function is \emph{almost}
constant in time.
Still, the function $u_2^\dagger$ has finite total variation
and thus falls into the theoretical setting considered here.

For the numerical solution of the PDE, we have used a semi-implicit
Crank--Nicolson method.
The data $u_i^\dagger$ as well as the corresponding solution $y_i = \mathcal{A}(u_i^\dagger)$
of the forward model are shown in Figure~\ref{fi:forward1}.

\begin{figure}
  \[
    \begin{aligned}
      \includegraphics[width=0.45\textwidth]{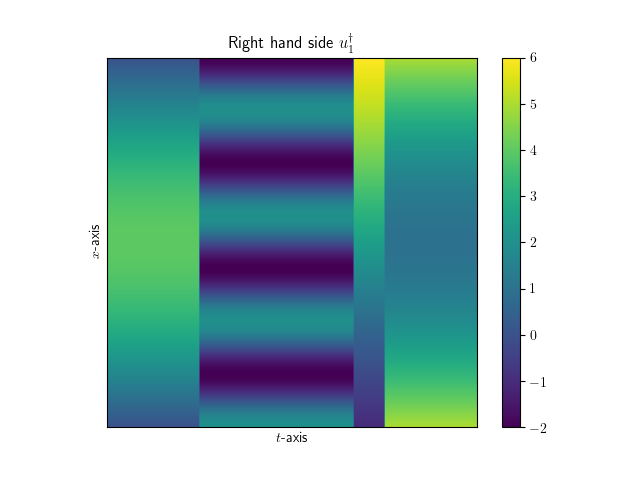}
      &\includegraphics[width=0.45\textwidth]{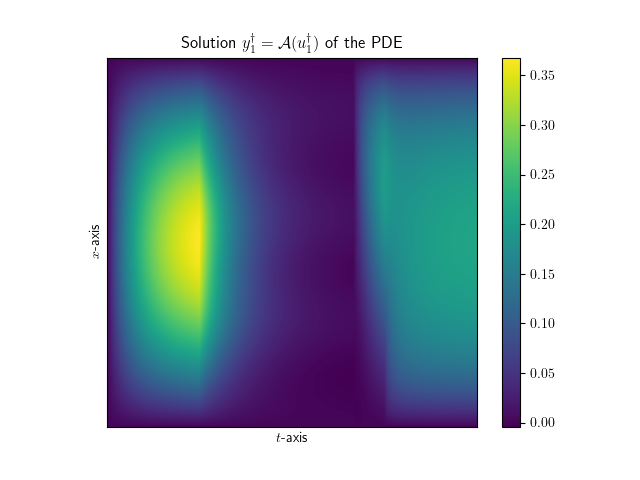}\\
      \includegraphics[width=0.45\textwidth]{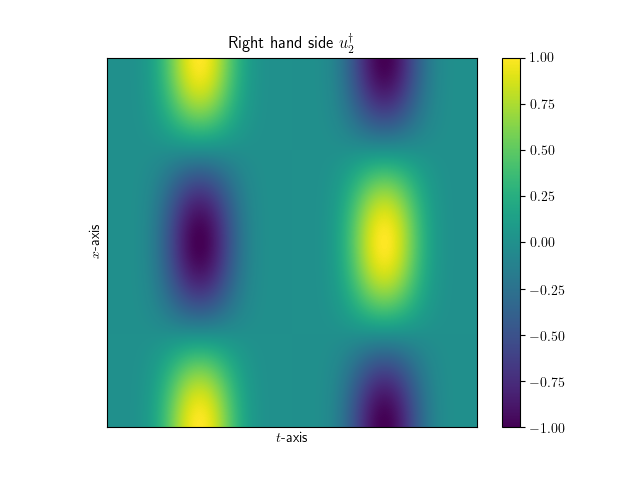}
      &\includegraphics[width=0.45\textwidth]{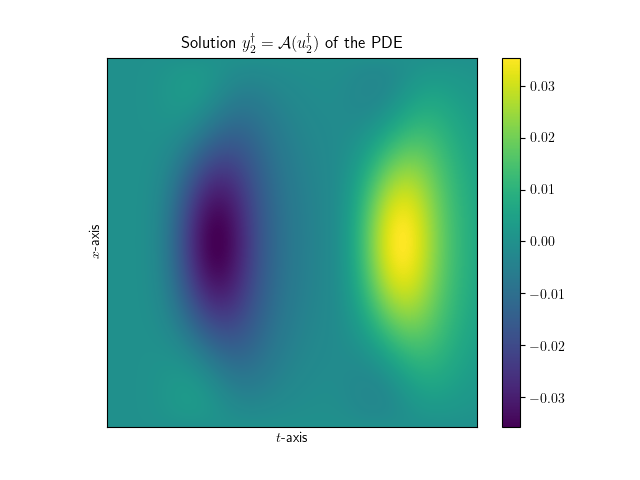}
    \end{aligned}
  \]
  \caption{\label{fi:forward1}
    Solution of the PDE~\eqref{eq:PDE}.
    \emph{First row, left:} right hand side $u = u_1^\dagger$ as defined in~\eqref{eq:u1}.
    \emph{First row, right:} corresponding solution $y_1 := \mathcal{A}(u_1^\dagger)$.
    \emph{Second row, left:} right hand side $u = u_2^\dagger$ as defined in~\eqref{eq:u2}.
    \emph{Second row, right:} corresponding solution $y_2 := \mathcal{A}(u_2^\dagger)$.}
\end{figure}

For the numerical tests of the algorithm and the regularization
method, we have generated noisy data $n^\delta_i$ by sampling
from an i.i.d.~Gaussian random variable with mean $0$ and
standard deviation $\sigma = \delta \lVert y_i \rVert_{L^2}$
and then defined $y_i^\delta = y_i + n_i^\delta$.
Thus the noise level $\delta$ always refers to the relative noise level
as compared to the true data $y_i = \mathcal{A}(u^\dagger)$.

Figure~\ref{fi:regsol} shows the reconstructions we obtain
for a noise level $\delta = 10^{-2}$ by solving~\eqref{eq:problem}.
For the true solution $u_1^\dagger$, we have set the regularisation parameters
to $\lambda = 10^{-4}$ and $\mu = 10^{-5}$;
for the true solution $u_2^\dagger$ to $\lambda = 2\cdot 10^{-6}$ and $\mu = 10^{-5}$.
The general shape of the true solution is well reconstructed, and
the method is also able to reconstruct the jumps in the function $u_1^\dagger$,
although the position of the jumps is not detected precisely.
Also, the error is relatively large at the boundary of the domain,
which can be explained by the fact that we solve the PDE~\eqref{eq:PDE}
with Dirichlet boundary conditions. Thus the function $u^\dagger$ has
only very little influence on $y^\dagger$ near the boundary,
which makes it hard to reconstruct.

\begin{figure}
  \[
    \begin{aligned}
      \includegraphics[width=0.45\textwidth]{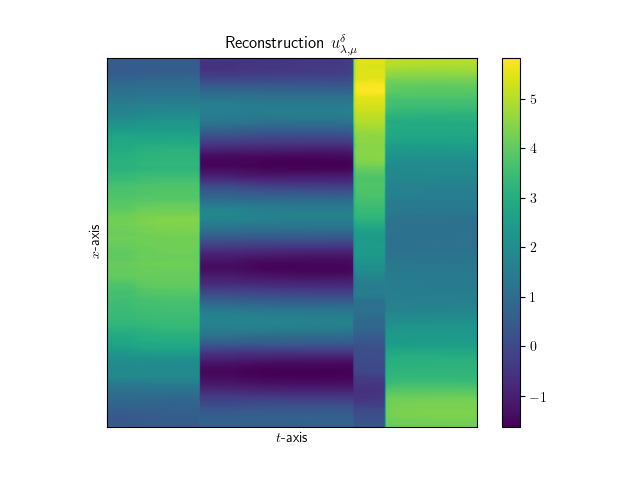}
      &\includegraphics[width=0.45\textwidth]{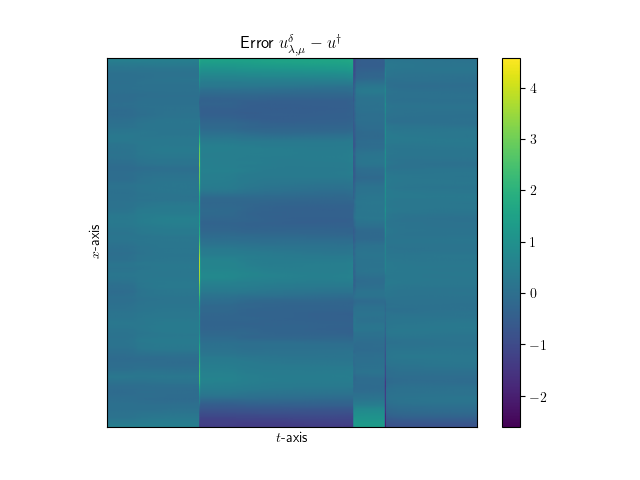}\\
      \includegraphics[width=0.45\textwidth]{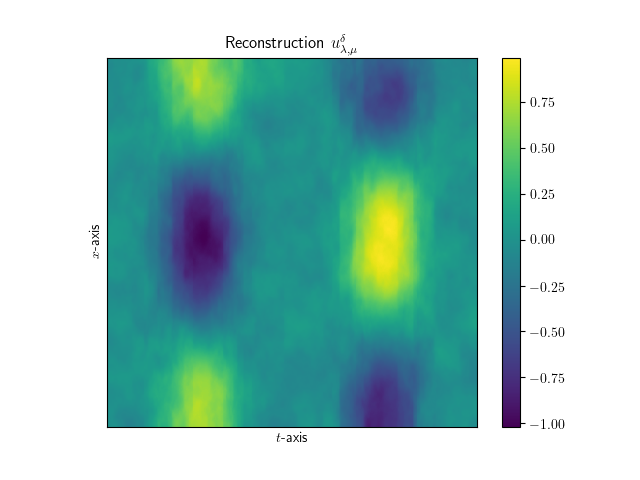}
      &\includegraphics[width=0.45\textwidth]{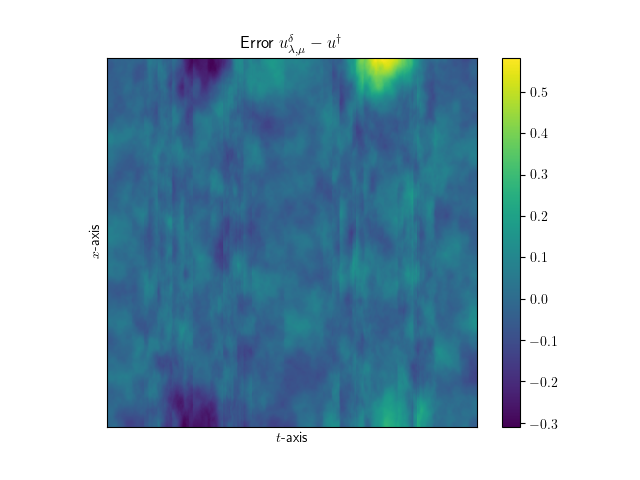}
    \end{aligned}
  \]
  \caption{\label{fi:regsol}
    Regularised solution of~\eqref{eq:ip} by the solution of~\eqref{eq:problem}.
    \emph{First row, left:} regularised solution $u_{\lambda,\mu}^\delta$ with noisy data
    $y^\delta = u_1^\dagger + n^\delta$ with $u_1^\dagger$ as defined in~\eqref{eq:u1},
    noise level $\delta = 10^{-2}$,
    and regularisation parameters $\lambda = 10^{-4}$ and $\mu = 10^{-5}$.
    \emph{First row, right:} resulting error $u_{\lambda,\mu}^\delta - u^\dagger$.
    \emph{Second row, left:} regularised solution $u_{\lambda,\mu}^\delta$ with noisy data
    $y^\delta = u_2^\dagger + n^\delta$ with $u = u_2^\dagger$ as defined in~\eqref{eq:u2},
    noise level $\delta = 10^{-2}$,
    and regularisation parameters $\lambda = 2\cdot 10^{-6}$ and $\mu = 10^{-5}$.
    \emph{Second row, right:} resulting error $u_{\lambda,\mu}^\delta - u^\dagger$.
  }
\end{figure}

\subsection*{Convergence of the algorithm}

In order to demonstrate the convergence properties of Algorithm~\ref{alg:applied},
we have applied the algorithm to a noisy version $y_1^\delta = y_1+n^\delta$
of $y_1$ with a noise level $\delta = 0.01$.
For the regularization parameters we chosen the values $\lambda = 10^{-5}$
and $\mu = 2\cdot 10^{-6}$. The number of iterations for the inner loop
in Algorithm~\ref{alg:applied} was set to $k_{\text{max}} = 5$.

Figure~\ref{fi:convergence_behaviour} (upper left) shows how the size of the updates
$\lVert u^{\textrm{old}}-u\rVert$ changes over the iterations.
We see how these step lengths roughly decrease linearly with the number of steps.
In addition, we see in~\ref{fi:convergence_behaviour} (upper right)
how the different term $\lVert \mathcal{A}(u) - y^\delta\rVert$,
$\lambda\mathcal{R}(u)$, $\mu\mathcal{S}(u)$ change with the iterations.

Next we recall that we are solving the inclusion~\eqref{eq:pdinclusion},
which for our problem reads
\[
  \begin{aligned}
    \mathcal{A}(u) -y^\delta + \mu\partial\mathcal{S}_\Gamma(u) + D_\Gamma^* v &= 0,\\
    v &\in \lambda\partial\mathcal{R}_\Gamma(D_\Gamma u).
  \end{aligned}
\]
Thus we can measure the convergence by assessing to which extent this equation
and inclusion are satisfied.
In the case of the inclusion $v \in \lambda \partial\mathcal{R}_\Gamma(D_\Gamma u)$
we exploit the fact that $\mathcal{R}_\Gamma$ is positively homogeneous.
As a consequence of this, the inclusion is equivalent to
\[
  \mathcal{R}_\Gamma^*(v) < \infty \qquad\text{ and }\qquad \langle v, D_\Gamma u\rangle = \mathcal{R}_\Gamma(D_\Gamma u).
\]
Since the inequality $\mathcal{R}_\Gamma^*(v) < \infty$ is automatically satisfied
in the algorithm, it thus makes sense to assess the convergence of the algorithm
using the values
\begin{equation}\label{eq:primal_and_dual_residual}
  \begin{aligned}
    r_1(u,v) &:= \lVert \mathcal{A}(u) -y^\delta + \mu\partial\mathcal{S}_\Gamma(u) + D_\Gamma^* v \rVert,\\
    r_2(u,v) &:= \lvert \mathcal{R}_\Gamma(D_{\Gamma} u) - \langle v,D_\Gamma u\rangle \rvert,
  \end{aligned}
\end{equation}
which we will refer to as the ``primal residual'' $r_1$ and the ``dual residual'' $r_2$.
The lower row in Figure~\ref{fi:convergence_behaviour} shows how these residuals
decrease with the iterations.

\begin{figure}
  \[
    \begin{aligned}
      &\includegraphics[width=0.3\textwidth]{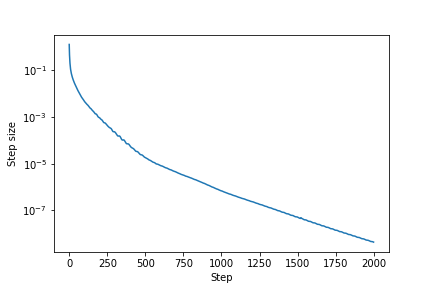} &
      &\includegraphics[width=0.3\textwidth]{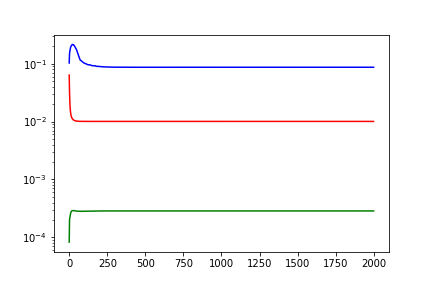} \\
      &\includegraphics[width=0.3\textwidth]{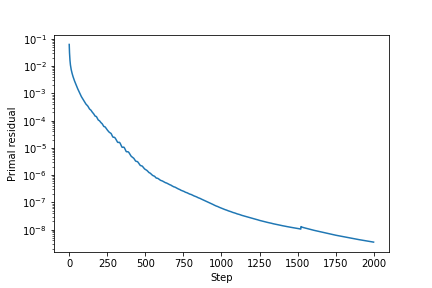} &
      &\includegraphics[width=0.3\textwidth]{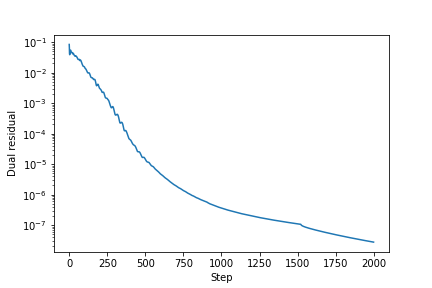}
    \end{aligned}
  \]
  \caption{\label{fi:convergence_behaviour}
    Demonstration of the convergence behavior of Algorithm~\ref{alg:applied}.
    \emph{Upper left:} Size of the updates $\lVert u^{(\text{old})}-u\rVert_2$.
    \emph{Upper right} Values of $\lVert \mathcal{A}(u)-y^\delta\rVert_2$ (red),
    the total variation regularization term $\lambda \mathcal{R}(u)$ (blue),
    and the $H^1$ regularization term $\mu\mathcal{S}(u)$ (green).
    \emph{Lower left:} Value of the primal residual $r_1(u,v)$.
    \emph{Lower right:} Value of the dual residual $r_2(u,v)$.
  }
\end{figure}

\subsection*{Semi-convergence of the regularization method}

Next, we study the behavior of the regularized solutions as
noise levels and regularization parameters simultaneously tend to $0$.
According to Theorem~\ref{th:well_posed}, the regularized solutions
converge in this case to the true solution, provided that the regularization parameters
tend to $0$ in such a way that the ratios $\lambda/\delta$ and $\mu/\delta$
remain bounded.

In order to verify this result numerically,
we have selected initial regularization parameters
$\lambda_0$, $\mu_0 > 0$, and an initial noise level $\delta_0$.
We have then applied our method with parameters
$\lambda_i = 2^{-i} \lambda_0$, $\mu_i = 2^{-i} \mu_0$,
and noise level $\delta_i = 2^{-i} \delta_0$ to
the two true solutions $u_1^\dagger$ and $u_2^\dagger$
defined in~\eqref{eq:u1} and~\eqref{eq:u2}.

Figure~\ref{fi:semiconvergence} shows convergence plots for these experiments.
We see there that the error $\lVert u-u_i^\dagger \rVert_{L^2}$
for both of the true solutions $u_1^\dagger$ and $u_2^\dagger$
roughly behaves like $O(\delta^{1/2})$
whereas the residual $\lVert \mathcal{A}(u)-u_i^\dagger\rVert_{L^2}$
behaves like $O(\delta)$.
This is in agreement with the theory for standard Lavrentiev regularization
for linear operators, where it is known that $O(\delta^{1/2})$
is the best possible rate for the error in non-trivial situations, see \cite{Pla17}.
We note here, though, that as of now there exist no theoretical results
concerning convergence rates with respect to the norm for our method.
In~\cite{GraHil20}, convergence rates with respect to the
Bregman distance defined by $\mathcal{R}+\mathcal{S}$ have
been derived. Because this functional is not strictly convex,
let alone $p$-convex, these rates cannot be translated to rates
in the $L^2$-norm, though.
Also, it is not clear whether the variational source condition required
in~\cite{GraHil20} holds for the test functions used for our experiments.

\begin{figure}
  \[
    \begin{aligned}
      \includegraphics[width=0.4\textwidth]{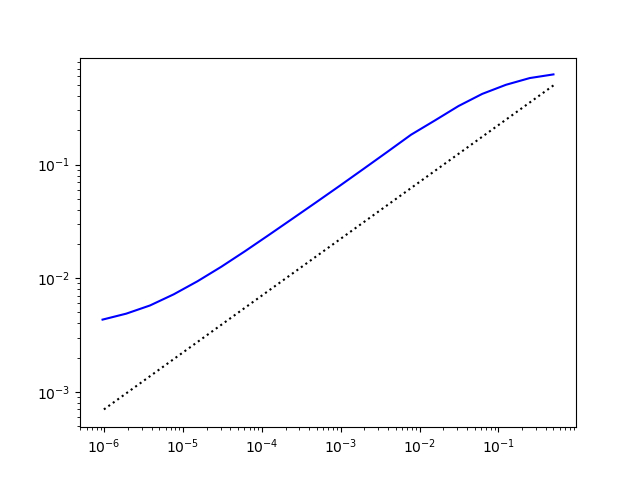}
      &\includegraphics[width=0.4\textwidth]{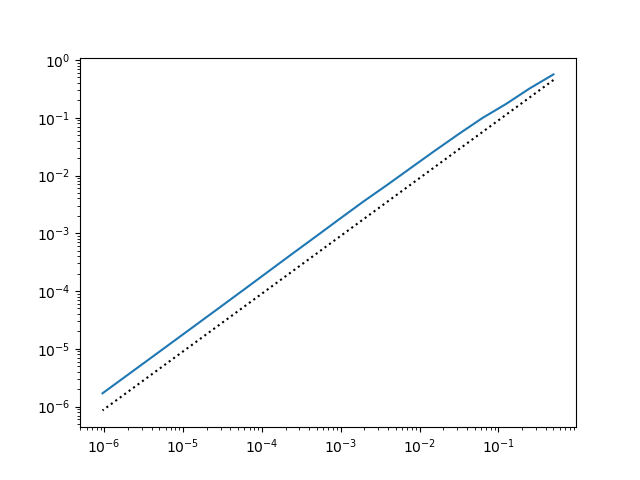}\\
      \includegraphics[width=0.4\textwidth]{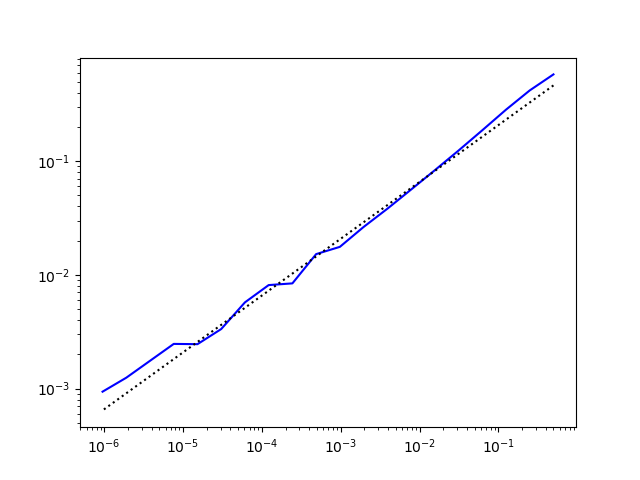}
      &\includegraphics[width=0.4\textwidth]{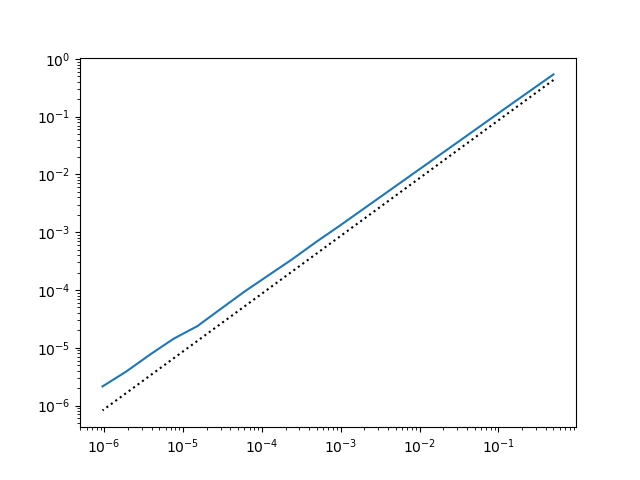}
    \end{aligned}
  \]
  \caption{\label{fi:semiconvergence}
    Convergence of the regularization method.
    \emph{Left:} Plot of the relative error $\lVert u-u_1^\dagger \rVert_2/\lVert u_1^\dagger\rVert_2$ (blue). The black dotted line indicates a rate of order $\sqrt{\delta}$.
    \emph{Right:} Plot of the relative residual $\lVert \mathcal{A}(u)-y_1\rVert_2/\lVert y_1\rVert_2$ (blue). The black dotted line indicates a rate of order $\delta$.
  }
\end{figure}

\subsection*{Comparison with different algorithms}

Finally, we show how Algorithm~\ref{alg:one}, again with $k_{\max} = 5$,
compares with existing algorithms for the solution of problems of the form~\eqref{eq:pdinclusion}.
Specifically, we consider the following two algorithms:
\begin{itemize}
\item The fixed point iteration (FP) for~\eqref{eq:lemma}
  \[
    \begin{aligned}
      u_{n+1} &= \prox_{\alpha g}\bigl(u_n - \alpha\bigl(\mathcal{T}(u_n) - L^* v_n\bigr)\bigr),\\
      v_{n+1} &= \prox_{\beta\alpha^{-1} f^*}\bigl(v_n + \beta\alpha^{-1} Lu_{n+1}\bigr).
    \end{aligned}
  \]
\item The inertial primal-dual forward-backward algorithm (IPDFB, see~\cite[eq.~(21)]{LorenzPock2015})
  \[
    \begin{aligned}
      \hat{u}_n &= u_n + \gamma_n (u_n - u_{n-1}),\\
      \hat{v}_n &= v_n + \gamma_n (v_n - v_{n-1}),\\
      u_{n+1} &= \prox_{\alpha g} \bigl(\hat{u}_n - \alpha\bigl(\mathcal{T}(\hat{u}_n) - L^* \hat{v}_n\bigr),\\
      \tilde{u}_{n+1} &= 2u_{n+1} - \hat{u}_n,\\
      v_{n+1} &= \prox_{\beta\alpha^{-1} f^*}\bigl(\hat{v}_n + \beta\alpha^{-1} L\tilde{u}_{n+1}\bigr).
    \end{aligned}
  \]
\end{itemize}

Compared to FP and IPDFB, each (outer) step of Algorithm~\ref{alg:one} is significantly
more computationally demanding, as it requires $k_{\max}$ evaluations of the the prox-operators for
$g$ and $f^*$. However, for the particular problem we are solving, these operations
can be fully vectorized and thus implemented rather efficiently.
The most expensive operation in each step of every algorithm is the evaluation of
the monotone operator $\mathcal{T}$, that is, the solution of the PDE~\eqref{eq:PDE},
for which the same degree of vectorization is not possible.
As a consequence, each (outer) step of Algorithm~\ref{alg:one} is in practice,
depending on the size of the problem, only twice as expensive as a step of FP or IPDFB.

Figure~\ref{fi:comp_time} shows a plot of the computation time versus accuracy for the
different algorithms and for different discretizations.
Again, we measure the accuracy in terms of the primal and dual residuals $r_1$ and $r_2$
defined in~\eqref{eq:primal_and_dual_residual}.
We stopped the iteration for each of the different algorithms as soon
as the corresponding residual dropped below $10^{-6}$.
We see that the primal residuals $r_1$ drop at a similar speed as those
for IPDFB.
The convergence of the dual residuals $r_2$ is significantly faster for Algorithm 1
compared to the other methods,
likely due to the fact that $k_{\max} = 5$ updates of the dual variable $v$
are performed in each (outer) step of Algorithm~\ref{alg:one}.

\begin{figure}
  \[
    \begin{aligned}
      &\includegraphics[width=0.32\textwidth]{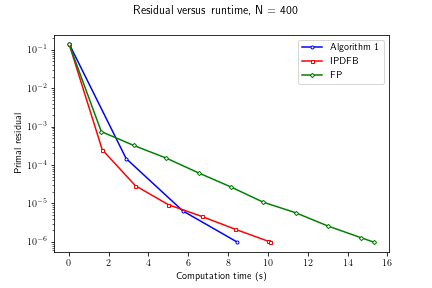}&
      &\includegraphics[width=0.32\textwidth]{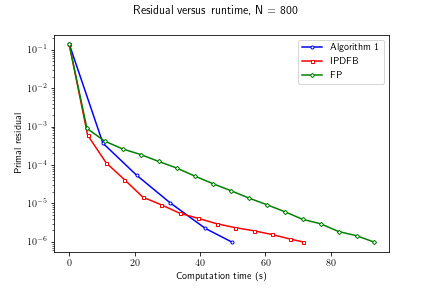}&
      &\includegraphics[width=0.32\textwidth]{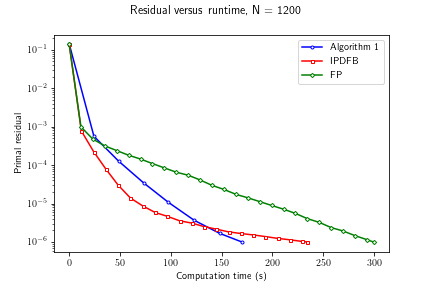}\\
      &\includegraphics[width=0.32\textwidth]{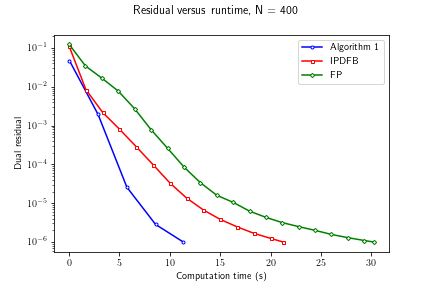}&
      &\includegraphics[width=0.32\textwidth]{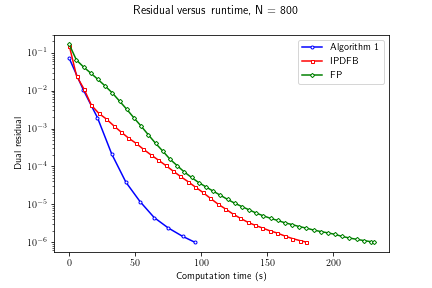}&
      &\includegraphics[width=0.32\textwidth]{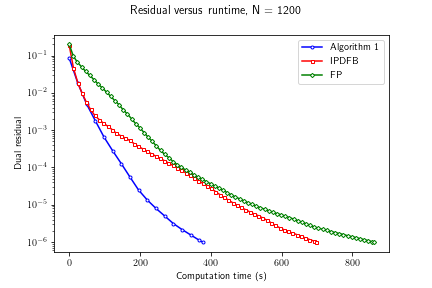}\\
    \end{aligned}
  \]
  \caption{\label{fi:comp_time}
    Comparison of accuracy versus computation times for the different algorithms
    and $N \in \{400,800,1200\}$ discretization points both for the $t$
    and the $x$ variable.
    Each marker on the plot represents 100 steps with the respective algorithm.
    Algorithm 1 is shown in blue in all graphs,
    the fixed point iteration in green,
    and the inertial primal-dual forward-backward algorithm in red.
    \emph{Top row:} Primal residuals for the different algorithms.
    \emph{Bottom row:} Dual residuals for the different algorithms.
  }
\end{figure}

\section*{Acknowledgment} The first author acknowledges the financial support from ERCIM ‘Alain Bensoussan’ Fellowship Programme. 



\end{document}